\documentclass[final]{siamart1116}

\usepackage[T1]{fontenc}
\usepackage{lipsum}
\usepackage{amsfonts}
\usepackage{graphicx}
\usepackage{epstopdf}
\usepackage{algorithmic}
\Crefname{ALC@unique}{Line}{Lines}
\usepackage{graphicx}
\usepackage{epstopdf}
\usepackage{url}
\usepackage{amsmath}
\usepackage{mathtools}
\usepackage{subcaption}
\usepackage{booktabs}
\usepackage{multirow}
\usepackage{cite}

\newsiamremark{example}{Example}
\newsiamthm{assumption}{Assumption}

\ifpdf
  \DeclareGraphicsExtensions{.eps,.pdf,.png,.jpg}
\else
  \DeclareGraphicsExtensions{.eps}
\fi

\numberwithin{theorem}{section}

\newcommand{\TheTitle}{A Thermodynamically Consistent Fractional 
Visco-Elasto-Plastic Model with Memory-Dependent Damage for Anomalous 
Materials} 
\newcommand{\TheAuthors}{Jorge Suzuki, Yongtao Zhou, Marta D'Elia, Mohsen 
Zayernouri}

\headers{Fractional V-E-P Model with Memory-Dependent Damage}{\TheAuthors}

\title{{\TheTitle}}

\author{
  Jorge L. Suzuki\thanks{Department of Mechanical Engineering and Department 
  of Computational Mathematics, Science and Engineering, Michigan State 
  University, East Lansing, MI 48824, USA
  	(\email{suzukijo@msu.edu}, \email{suzukijo@egr.msu.edu}).}
  \and
  Yongtao Zhou\thanks{School of Mathematics and Statistics, Huazhong 
  University of Science and Technology, Wuhan 430074, China and Department of 
  Mechanical Engineering, Michigan State University, East Lansing, MI 48824, USA
  	(\email{yongtaozh@126.com}).}
  \and
  Marta D'Elia\thanks{Center for Computing Research, Sandia National 
  Laboratories, P.O. Box 5800, MS 9159, Albuquerque, NM 87185, USA
  	(\email{mdelia@sandia.gov}).}
  \and
  Mohsen Zayernouri\thanks{Department of Mechanical Engineering and Department 
  of Statistics and Probability, Michigan State University, East Lansing, MI 
  48824, USA
  	(\email{zayern@msu.edu}, \email{zayern@egr.msu.edu}), Corresponding Author.}
}

\usepackage{amsopn}

\ifpdf
\hypersetup{
  pdftitle={\TheTitle},
  pdfauthor={\TheAuthors}
}
\fi

\newsiamthm{remark}{Remark}

\DeclareMathOperator{\sign}{\mathrm{sign}}

\begin{document}

\maketitle

\begin{abstract}
	We develop a thermodynamically consistent, fractional visco-elasto-plastic 
model coupled with damage for anomalous materials. The model utilizes 
Scott-Blair rheological elements for both visco- elastic/plastic parts. The 
constitutive equations are obtained through Helmholtz free-energy potentials 
for Scott-Blair elements, together with a memory-dependent fractional yield 
function and dissipation inequalities. A memory-dependent Lemaitre-type damage 
is introduced through fractional damage energy release rates. For 
time-fractional integration of the resulting nonlinear system of equations, we 
develop a first-order semi-implicit 
fractional return-mapping algorithm. We also develop a finite-difference 
discretization for the fractional damage energy release rate, which results 
into Hankel-type matrix-vector operations for each time-step, allowing us to 
reduce the computational complexity from $\mathcal{O}(N^3)$ to 
$\mathcal{O}(N^2)$ through the use of Fast Fourier Transforms. Our numerical 
results demonstrate that the fractional orders for visco-elasto-plasticity play 
a crucial role in damage evolution, due to the competition between the 
anomalous plastic slip and bulk damage energy release rates.
\end{abstract}



\section{Introduction}

Accurate and predictive modeling of material damage and failure for a wide 
range of materials poses multi-disciplinary challenges on experimental 
detection, consistent physics-informed models and efficient 
algorithms. Material failure arises in mechanical and 
biological systems as a consequence of internal damage, 
characterized in the micro-scale by the presence and growth of discontinuities 
\textit{e.g.}, microvoids, microcracks and bond breakage. Continuum Damage 
Mechanics (CDM) treats such effects in the macroscale through a representative 
volume element (RVE) \cite{Lemaitre2005}. When loading plastic crystalline 
materials, an initial hardening stage is observed from motion, arresting and 
network formation of dislocations, which is later overwhelmed by damage 
mechanisms, \textit{e.g.} multiplication of micro-cracks/voids, followed by 
their growth and coalescing, releasing bulk energy from the RVE. Classical CDM 
models were proposed {and validated} in the past decades to describe the 
mechanical degradation, {\textit{e.g.}}, of 
ductile, brittle, and hyperelastic materials \cite{Simo1987,Lemaitre1996}. 
Particularly, Lemaitre's ductile damage model \cite{Lemaitre1996,Lemaitre2005} 
has been {extensively employed for} plasticity and visco-plasticity modeling of 
ductile {materials}. In such models, developing proper damage potentials driven 
by the {so-called \textit{damage energy release rate} \cite{Lemaitre2005} is a 
	critical step.}

{Modeling the standard-to-anomalous damage evolution for power-law materials 
	has  additional challenges due to the non-Gaussian processes occurring on} 
fractal-like media. Fractional constitutive laws {utilize Scott-Blair (SB)} 
elements \cite{Blair1943,Blair1947} as 
rheological building blocks {that model the soft material response as a 
	power-law} memory-dependent device, interpolating
between purely elastic/viscous behavior. {A mechanical representation of the SB 
	element was developed by Schiessel \cite{Schiessel1993}, as a hierarchical, 
	continuous ``ladder-like" 
	arrangement of {canonical Hookean/Newtonian} elements (\textit{see Figure 
		\ref{fig:SB}}).} Later on, Schiessel \cite{Schiessel1995} 
generalized several standard visco-elastic models (Kelvin-Voigt, Maxwell, 
Kelvin-Zener, Poynting-Thompson) to their fractional 
counterparts by fully replacing the {canonical} elements with SB elements. Of 
particular interest, Lion \cite{Lion1997} {proved the thermodynamic consistency 
	of the SB element from a mechanically-based fractional Helmholtz 
	free-energy 
	density.}

With particular arrangements of SB and {standard} elements, fractional models 
were applied, \textit{e.g.}, to 
describe the \textit{far from equilibrium} power-law dynamics of 
multi-fractional visco-elastic 
\cite{Jaishankar2013,Magin2010,Nasholm2013Wave,Giusti2017,
	naghibolhosseini2015estimation,Naghibolhosseini2018}, {distributed} 
visco-elastic \cite{Dekic2017} and visco-elasto-plastic 
\cite{Sumelka2014VP,Suzuki2016,Xiao2017,Hei2018,Sumelka2019Soil} 
complex materials. Concurrently, significant 
advances in numerical methods allowed {numerical solutions to} time- and space- 
fractional {partial differential equations (FPDEs)} {for} smooth/non-smooth 
solutions, such as finite-difference (FD) schemes \cite{Lubich1986,Lin2007}, 
fractional Adams methods \cite{Diethelm2004,Zayernouri2016MS}, 
implicit-explicit (IMEX) 
schemes \cite{Cao2016,Zhou2019IMEX}, spectral methods 
\cite{Samiee2016,samiee2017unified}, fractional subgrid-scale modeling 
\cite{Samiee2019FSSM}, fractional sensitivity {equations} 
\cite{kharazmi2017FSEM}, operator-based uncertainty 
quantification \cite{kharazmi2018operatorbased} and self-singularity-capturing 
approaches \cite{Suzuki2018Singularity}. 

Despite the {significant contributions on} fractional constitutive laws, {few 
	works incorporated damage mechanisms. Zhang \textit{et al.} 
	\cite{Zhang2014creep} developed a nonlinear, visco-elasto-plastic creep 
	damage 
	model for concrete, where the damage evolution was defined through an 
	exponential function of time.} A similar model was proposed by Kang 
	\textit{et 
	al.} \cite{Kang2015} {and applied to coal creep. 
	Caputo and Fabrizio \cite{Caputo2015} developed a variable order 
	visco-elastic model, where the variable order was regarded as a phase-field 
	driven damage. Alfano and Musto \cite{Alfano2017} 
	developed a cohesive zone, damaged fractional Kelvin-Zener model, and 
	studied the influence of Hooke/SB damage energy release rates on damage 
	evolution, motivating further studies on crack propagation mechanisms in 
	visco-elastic media. Tang \textit{et al.} \cite{Tang2018} 
	developed a variable order rock creep model, with damage evolution as an 
	exponential function of time.  Recently, Giraldo-Londo\~{n}o \textit{et 
	al.} 
	\cite{Giraldo-Londono2019} developed a two-parameter, two-dimensional (2-D) 
	rate-dependent cohesive fracture model.}

{A key aspect to develop failure models relies on consistent forms of damage 
	energy release rates, usually appearing in the material-specific form of 
	Helmholtz free-energy densities. For standard materials, direct summations 
	of 
	elastic/hyperelastic free-energies of the system are used. However, such 
	process is non-trivial when modeling anomalous materials, due to the 
	intrinsic 
	mixed elasticity/viscosity of SB elements. Fabrizio 
	\cite{Fabrizio2014} introduced a Graffi-Volterra free-energy for fractional 
	models, but defined it without sufficient physical justification. Deseri 
	\textit{et al.} \cite{Deseri2014} developed free-energies for fractional 
	hereditary materials, with the notion of order-dependent 
	\textit{elasto-viscous} and \textit{visco-elastic} behaviors. Lion 
	\cite{Lion1997} derived the isothermal Helmholtz free-energy density for SB 
	elements using a discrete-to-continuum arrangement of standard Maxwell 
	branches, and employed it in the Clausius-Duhem inequality to obtain the 
	stress-strain relationship. Later on, Adolfsson \textit{et al.} 
	\cite{Adolfsson2005} employed Lion's 
	approach to prove the thermodynamic admissibility of the SB constitutive 
	law written as a Volterra integral equation of first kind.}

To the authors' best knowledge, only Alfano and Musto \cite{Alfano2017} 
coupled the fractional free-energy density to a damage evolution equation in 
viscoelasticity, but fractional extensions of (non-exponential) damage for 
visco-elasto-plastic materials are still lacking. In addition, for 
damage models, efficient numerical methods for fractional free-energy 
computations are also virtually nonexistent in the literature. A numerical 
approximation was done by Burlon \textit{et al.} \cite{Burlon2014}, 
through a finite summation of free-energies from Hookean elements, which is a 
truncation of the infinite number of relaxation modes carried by the fractional 
operators. Alfano and Musto \cite{Alfano2017} briefly described how to 
discretize the SB free-energy using a midpoint finite-difference scheme. A few 
numerical results were presented for damage evolution, but the authors did not 
describe the discretizations and no accuracy is investigated for the numerical 
scheme.

In this work we develop a thermodynamically consistent, one-dimensio\-nal (1-D)
fractional visco-elasto-plastic model with memory-dependent da\-mage in the 
context of CDM. The main characteristics of the model follow:
\begin{itemize}
	\item We employ SB elements in both visco-elastic and 
	visco-plastic parts, {respectively, with orders $\beta_E,\beta_K \in 
		(0,1)$, leading to power-law effects in both ranges.}
	
	\item The damage reduces the total free-energy of the model, {while 
		constitutive laws} are obtained through the Clausius-Duhem inequality.
	
	\item The yield function is time-fractional rate-dependent, while the 
	damage potential is Lemaitre-like. {The damage energy release rate is taken 
		as the SB Helmholtz free-energy density to describe the anomalous bulk 
		energy loss.}
	
	\item We prove the positive dissipation, and therefore the thermodynamic 
	consistency of the developed model (\textit{see Theorem 
		\ref{thm:positive_D}}).
\end{itemize}

{Since obtaining analytical solutions for the resulting nonlinear system of 
	multi-term visco-elasto-plastic fractional differential equations (FDEs) 
	coupled with damage} is cumbersome or even impossible, we performed an 
efficient time-integration framework as follows:
\begin{itemize}
	\item We develop a first-order, semi-implicit fractional return-mapping 
	algorithm, {with explicit evaluation of damage in the stress-strain 
		relationship and yield function. An implicit FD scheme is employed to 
		the 
		ODEs for plastic and damage variables. The time-fractional 
		stress-strain 
		relationship and yield function are discretized using the L1 FD scheme 
		from 
		Lin and Xu \cite{Lin2007}.}
	
	\item {We develop a fully-implicit scheme for the SB Helmholtz free-energy 
		density, and hence to the fractional damage energy release rate. We 
		then 
		exploit the structure of the discretized} energy and apply Fast Fourier 
	Transforms (FFTs) to obtain an efficient scheme.
	
	\item The accuracy of free-energy discretization is proved to 
	be of order $\mathcal{O}(\Delta t^{2-\beta})$, and numerical tests 
	{show} a computational complexity of order $\mathcal{O}(N^2 \log N)$, with 
	N being the number of time-steps.
\end{itemize}

{The developed fractional return-mapping algorithm can be easily incorporated 
	to existing finite element (FE) frameworks as a constitutive box. Numerical 
	tests are performed with imposed monotone and cyclic strains, and 
	demonstrate 
	that:}
\begin{itemize}
	\item {Softening, hysteresis and low-cycle fatigue can be modeled.}
	
	\item {Memory-dependent damage energy release rates induce anomalous 
		damage evolutions with competing visco-elastic/plastic effects, without 
		changing the form of Lemaitre's damage potential.}
\end{itemize}

{The developed model motivates applications to failure of biological materials 
	\cite{Bonadkar2016}, where micro-structural 
	evolution can be upscaled to the continuum through evolving fractional 
	orders $\beta_E$, $\beta_K$ \cite{Mashayekhi2019Fractal} and 
	damage $D$. The memory-dependent fractional damage energy release rates 
	motivate studies on anomalous bulk-to-surface energy loss in damage 
	accumulation/crack propagation of, \textit{e.g.}, bone tissue, where 
	intrinsic/extrinsic plasticity/crack-bridging mechanisms \cite{Wegst20151} 
	lead to a complex nature of failure.}

This work is organized as follows: In Section \ref{Sec:Definitions} {we present 
	definitions of fractional operators}. In Section \ref{Sec:SB}, we present 
	the 
thermodynamics and rheology of SB elements. In Section \ref{Sec:VEPD}, {we 
	develop the fractional visco-elasto-plastic model with damage, followed by 
	its 
	discretization}. A series of numerical tests are shown in Section 
\ref{Sec:NumericalResults}, followed by discussions and concluding remarks in 
Section \ref{Sec:Conclusions}.

\section{Definitions of Fractional Calculus}
\label{Sec:Definitions}
%

We start with some preliminary definitions of fractional calculus 
\cite{Podlubny99}. The left-sided Riemann-Liouville integral of order {$\beta 
	\in (0,1)$} is defined as
\begin{equation}
\label{Eq: left RL integral}
(\prescript{RL}{t_L}{\mathcal{I}}_{t}^{\beta} f) (t) = \frac{1}{\Gamma(\beta)} 
\int_{t_L}^{t} \frac{f(s)}{(t - s)^{1-\beta} }\, ds,\,\,\,\,\,\, t>t_L, 
\end{equation}
where $\Gamma$ represents the Euler gamma function and $t_L$ denotes the lower 
integration limit. The corresponding inverse operator, i.e., the left-sided 
fractional derivative of order $\beta$, is then defined based on (\ref{Eq: left 
	RL integral}) as  
\begin{equation}
\label{Eq: left RL derivative}
(\prescript{RL}{t_L}{\mathcal{D}}_{t}^{\beta} f) (t) = \frac{d}{dt} 
(\prescript{RL}{t_L}{\mathcal{I}}_{t}^{1-\beta} f) (t) = 
\frac{1}{\Gamma(1-\beta)}  \frac{d}{dt} \int_{t_L}^{t} \frac{f(s)}{(t - 
	s)^{\beta} }\, ds,\,\,\,\,\,\, t>t_L. \nonumber
\end{equation}
{Also, the left-sided Caputo derivative of order $\beta 
	\in (0,1)$ is obtained as}
\begin{equation}
\label{Eq: left Caputo derivative}
(\prescript{C}{t_L}{\mathcal{D}}_{t}^{\beta} f) (t) = 
(\prescript{RL}{t_L}{\mathcal{I}}_{t}^{1-\beta} \frac{df}{dt}) (t) = 
\frac{1}{\Gamma(1-\beta)}  \int_{t_L}^{t} \frac{f^{\prime}(s)}{(t - s)^{\beta} 
}\, ds,\,\,\,\,\,\, t>t_L. \nonumber
\end{equation}
{The definitions of Riemann-Liouville and Caputo derivatives are linked by the 
	following relationship:
	\begin{equation}
	\label{Eq: Caputo vs. Riemann}
	(\prescript{RL}{t_L}{\mathcal{D}}_{t}^{\beta} f) (t)  =  
	\frac{f(t_L)}{\Gamma(1-\beta) (t+t_L)^{\beta}}  +   
	(\prescript{C}{t_L}{\mathcal{D}}_{t}^{\beta} f) (t), \nonumber
	\end{equation}
	which denotes that the definition of the aforementioned derivatives 
	coincide 
	when dealing with homogeneous Dirichlet initial/boundary conditions.}

\subsection{Interpretation of Caputo derivatives in terms of nonlocal 
	vector calculus}
\label{Section:fractional-limit}
In this section we show that the Caputo derivative can be reinterpreted as the 
limit of a nonlocal truncated time derivative \cite{Du2017}. This fact 
establishes a connection between nonlocal initial value problems and their 
fractional counterparts, which can benefit from the nonlocal theory.

Given a nonnegative and symmetric kernel function 
$\rho_\delta(s)=\rho_\delta(|s|)$, a nonlocal, weighted, gradient operator can 
be defined as \cite{Du2013}
\begin{equation}
\mathcal G_\delta f(t) = \lim\limits_{\epsilon\to 
0}\displaystyle\int_\epsilon^\delta
(f(t)-f(t-s))s\rho_\delta(s)\,ds,
\end{equation}
when the limit exists in $L^2(0,T)$ for a function $f\in L^2(0,T)$. It is 
common to assume that the kernel function $\rho_\delta$ has compact support in 
$[-\delta,\delta]$ and a normalized moment:
\begin{equation}
\displaystyle\int_0^\delta s^2\rho_\delta(s)\,ds = 1.
\end{equation}
Here, the parameter $\delta>0$ represents the extent of the nonlocal 
interactions or, in case of time dependence, the memory span. In the nonlocal 
theory it is usually referred to as {\it horizon}.

Note that at the limit of vanishing nonlocality, {\it i.e.} as $\delta\to 0$, 
$\mathcal G_\delta$ corresponds to the classical first order time derivative 
operator $\frac{d}{dt}$. In this work, we are interested in the limit of 
infinite interactions, {\it i.e.} as $\delta\to \infty$. Specifically, when the 
initial data $f(t):=f(0)$ for all $t\in(-\infty,0)$ and the kernel function is 
defined as
\begin{equation}
\rho_\infty(s)=\dfrac{\beta}{\Gamma(1-\beta)}s^{-\beta-2}, 
\quad {\rm for}\;\beta\in(0,1), 
\end{equation}
the nonlocal operator $\mathcal G_\delta$ corresponds to the Caputo fractional 
derivative for $t>0$, for a piecewise differentiable function $f\in 
C(-\infty,T)$ such that $f'\in L^1(0,T)\cap C(0,T]$. Formally,
\begin{equation}
\mathcal G_\infty f(t) = 
(\prescript{C}{0}{\mathcal D }_t^\beta f)(t).
\end{equation}
Note that a similar property holds true for fractional derivatives in space, 
see \cite{DElia2013}.
\subsubsection{Note on well-posedness} Paper \cite{Du2017} analyzes the 
well-posedness of nonlocal initial value problems. More specifically, it 
proves, under certain conditions on the parameters, that the following equation 
has a unique solution and depends continuously upon the data.
\begin{equation}
\begin{aligned}
\mathcal G_\delta \gamma + H \gamma &= F & \;\;t\in (0,T],\\[1mm]
\gamma &= G & \;\; t\in (-\delta,0),
\end{aligned}
\end{equation}
for $H>0$ and $F$ and $G$ in suitable functional spaces.

\section{Thermodynamics of Fractional Scott-Blair Elements}
\label{Sec:SB}

{We} present the thermodynamic principles used in this 
work, and then we introduce the Helmholtz free-energy density and 
constitutive law for the fractional SB element. {Such fractional element is 
	the rheological building block of our modeling approach, providing a 
	constitutive interpolation between a Hookean $(\beta \to 0)$ and Newtonian 
	$(\beta \to 1)$ element \textit{(see Figure \ref{fig:SB})}. Furthermore, 
	the 
	SB element can be interpreted as an infinite self-similar arrangement of 
	standard Maxwell elements, which naturally leads to fractional operators in 
	the constitutive law \cite{Schiessel1993}.}
\begin{figure}[t!]
	\centering
	\begin{subfigure}[b]{0.45\textwidth}
		\centering
		\includegraphics[width=\columnwidth]{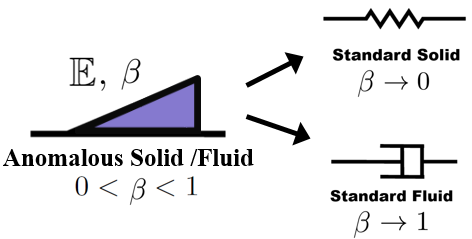}
	\end{subfigure}
	\hspace{1mm}
	\begin{subfigure}[b]{0.45\columnwidth}
		\centering
		\includegraphics[width=0.7\columnwidth]{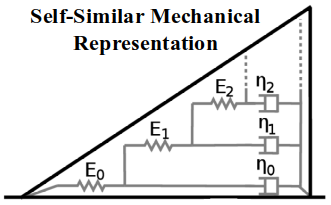}
	\end{subfigure}
	\caption{\textit{(left)} {Schematics of the SB element recovering standard 
			limit cases. \textit{(right)} The SB element seen as an infinite, 
			hierarchical mechanical representation of canonical elements, 
			coding an 
			infinite number of relaxation times. The pair $(\mathbb E,\beta)$ 
			represents a dynamic process of the material. \label{fig:SB}}}
\end{figure}

\subsection{Thermodynamic Principles}
\label{Sec:Thermodynamics}

Let a closed system {undergo} an irreversible, isothermal, 
strain-driven thermodynamic process. We analyze an infinitesimal material 
region {at a position $x$ and time $t$} of a continuum 
deformable body $\mathcal{B}$. {Let} the first law of thermodynamics in {rate} 
form \cite{Bejan2016} be defined as:
\begin{equation}\label{eq:first-law}
\dot{e} = \dot{q} - \dot{w},
\end{equation}
where $\dot{e}(x,t)$ $[J.s^{-1}.kg^{-1}]$ denotes the specific rate of 
internal energy, the term $\dot{q}(x,t)$ $[J.s^{-1}.kg^{-1}]$ represents the 
rate of 
specific heat 
exchange, and $\dot{w}(x,t)$  $[J.s^{-1}.kg^{-1}]$ denotes the 
{stress power transferred into the bulk due to 
	external} forces \cite{Gurtin2010}. In this work, $\tau(x,t)$ represents 
	the 
stress state and $\dot{\varepsilon}$ the strain rate. {We} also consider the 
second law of 
thermodynamics, postulating the irreversibility of entropy production, 
given, in specific form, by:
\begin{equation} \label{eq:second-law}
\dot{s} \ge \dot{q}/\theta,
\end{equation}
where $\dot{s}(x,t)\,[J.s^{-1}.kg^{-1}.K^{-1}]$ denotes the rate of specific 
entropy production and $\theta(x,t) = \theta_0\,[K]$ represents the 
{constant} temperature. Let 
$\psi(x,t):\mathbb{R}\times\mathbb{R}^+ \to \mathbb{R}^+$ be the 
Helmholtz free-energy density with units $[J.m^{-3}]$, representing the 
available energy to perform work, defined by $\psi := \rho\left(e - \theta 
s\right)$, with the rate form $\dot{\psi} = \rho\left(\dot{e} - \theta 
\dot{s}\right)$ {for the isothermal case}. Combining the first 
and second laws, respectively, (\ref{eq:first-law}) and (\ref{eq:second-law}), 
with $\dot{\psi}$ and {taking} the stress power $\dot{w} = 
-\tau \dot{\varepsilon}$, we obtain the Clausius-Duhem 
inequality, {which states the non-negative dissipation rates} 
\cite{deSouzaNeto2008}:
\begin{equation}\label{eq:CD-2}
-\dot{\psi} + \tau \dot{\varepsilon} \ge 0, \quad \forall x \in \mathcal{B}.
\end{equation}
{Satisfying the dissipation inequality 
	(\ref{eq:CD-2}) is here taken as the necessary condition for the potential 
	$\psi$ and the stress $\tau$ to be thermodynamically admissible.}

\subsection{Helmholtz Free-Energy Density}
\label{Section:Free-energy}

{We present the free-energy under consideration for the 
	employed SB element, here referred to a given material coordinate of a 
	continuum body or a lumped mechanical system. We start with the fractional 
	Helmholtz free-energy density developed by Lion \cite{Lion1997}, obtained 
	through an integration of a continuum spectrum of Maxwell branches leading 
	to 
	the following definition for $\psi(\varepsilon):\mathbb{R}\to\mathbb{R}^+$:}
\begin{equation}\label{eq:ve_potential}
\psi(\varepsilon) = \frac{1}{2} \int^\infty_0 \tilde{E}(z) \left[ \int^t_0 
\exp\left(-\frac{t-s}{z}\right) \dot{\varepsilon}(s)\,ds \right]^2\, dz,
\end{equation}
{where we the strain $\varepsilon$ is taken as the state 
	variable. The term 
	$\tilde{E}(z):\mathbb{R}^+\to\mathbb{R}^+$ denotes the power-law relaxation
	spectrum, given by}
\begin{equation}\label{eq:rel_spectrum}
\tilde{E}(z) = \frac{\mathbb{E}}{\Gamma(1-\beta) \Gamma(\beta) z^{\beta + 1}}, 
\quad 0 < \beta < 1, \quad \mathbb{E}\in\mathbb{R}^+, \nonumber
\end{equation}
which with (\ref{eq:ve_potential}), codes {an} infinite number of relaxation 
times. {The} pseudo-constant $\mathbb{E}$ has units $[Pa.s^{\beta}]$, where 
the unique pair $(\mathbb{E},\beta)$ {codes} a dynamic process instead of an 
equilibrium state of the material \cite{Jaishankar2013}. Let 
$\mathcal{D}_{mech}$ denote the mechanical dissipation of the SB element. We 
introduce the following Lemmas:

\begin{lemma}\label{lemma:stress_strain}
	{The SB element} stress-strain relationship 
	$\tau(t):\mathbb{R}^+\to\mathbb{R}$ 
	resulting from (\ref{eq:ve_potential}) and the Clausius-Duhem inequality
	(\ref{eq:CD-2}) is given by 
	\begin{equation}\label{eq:stress_strain_SB}
	\tau(t) = \int^\infty_0 \tilde{E}(z) \left( \int^t_0 
	\exp\left(-\frac{t-s}{z}\right) \dot{\varepsilon}(s)\,ds \right)\, dz =    
	\mathbb{E}\,{}^C_0 \mathcal{D}^\beta_t \varepsilon(t),
	\end{equation}
	where the Caputo definition for the fractional derivative is a consequence 
	of the adopted free-energy. {The} mechanical dissipation 
	$\mathcal{D}_{mech}(\varepsilon):\mathbb{R}\to\mathbb{R}^+$ for the SB 
	element is given by the following form:
	\begin{equation}\label{eq:dissipation_SB}
	\mathcal{D}_{mech}(\varepsilon) = \int^\infty_0 \frac{\tilde{E}(z)}{z} 
	\left( \int^t_0 \exp\left(-\frac{t-s}{z}\right) \dot{\varepsilon}(s)\,ds 
	\right)^2 dz.
	\end{equation}
\end{lemma}
{
	\begin{proof}
		See Appendix \ref{Ap:lemma_stress_strain}.
	\end{proof}
	\begin{remark}
		The limit cases for the fractional free-energy 
		(\ref{eq:ve_potential}) with respect to $\beta$ are consistent with the 
		well-known stress-strain relationship (\ref{eq:stress_strain_SB}). 
		Therefore, 
		$\psi(\varepsilon)$ recovers a fully conserving Hookean spring when 
		$\lim_{\beta \to 0} \psi = \mathbb{E} \varepsilon^2/2$, and a fully 
		dissipative Newtonian dashpot when $\lim_{\beta \to 1} \psi = 0$. We 
		refer the 
		readers to \cite{Lion1997, Deseri2014} for additional details regarding 
		memory-dependent free-energies.
\end{remark}}

\section{Fractional Visco-Elasto-Plastic Model with Damage}
\label{Sec:VEPD}

{We} develop a damage formulation for a fractional 
visco-elasto-plastic model (M1) by Suzuki \textit{et al.} \cite{Suzuki2016}.  
{The} closure for the damage variable is obtained through a 
Lemaitre-type approach \cite{Lemaitre1996,Lemaitre2005}. {We} 
{later} prove the thermodynamic consistency of the damage model, 
and {hence} for the visco-elasto-plastic model (M1) as a 
limiting, undamaged case.


\subsection{Thermodynamic Formulation}

The {fractional visco-elasto-plastic device} is 
illustrated in Figure \ref{fig:VEP_device}. It consists of a SB element with  
{material} pair $(\mathbb{E}, \beta_E)$ for the visco-elastic 
part, under 
a corresponding logarithmic visco-elastic strain 
$\varepsilon^{ve}(t):\mathbb{R}^+\to\mathbb{R}$. The visco-plastic part is 
given by a parallel combination of a Coulomb frictional element with yield 
stress $\tau^Y\,[Pa] \in \mathbb{R}^+$, a linear hardening Hooke element with 
constant $H\,[Pa] \in \mathbb{R}^+$, and a SB element with material pair 
$(\mathbb{K}, \beta_K)$, with $\mathbb{K}\, [Pa.s^{\beta_K}] \in \mathbb{R}^+$, 
all subject to a logarithmic visco-plastic strain 
$\varepsilon^{vp}(t):\mathbb{R}^+\to\mathbb{R}$ and an internal hardening 
variable $\alpha(t):\mathbb{R}^+\to\mathbb{R}^+$. The entire device is subject 
to a Kirchhoff stress $\tau$. The total logarithmic strain is 
{given by:}
\begin{equation}\label{eq:kinematics}
\varepsilon(t) = \varepsilon^{ve}(t) + \varepsilon^{vp}(t).
\end{equation}
\begin{figure}[t!]
	\centering
	\begin{subfigure}[b]{0.45\textwidth}
		\includegraphics[width=\columnwidth]{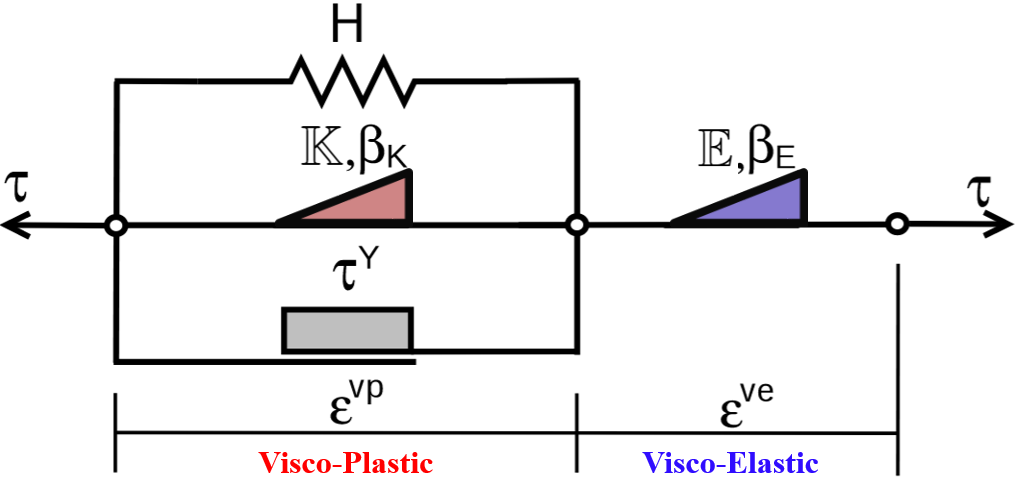}
		\caption{Rheological diagram.}		
	\end{subfigure}%
	\begin{subfigure}[b]{0.45\textwidth}
		\includegraphics[width=\columnwidth]{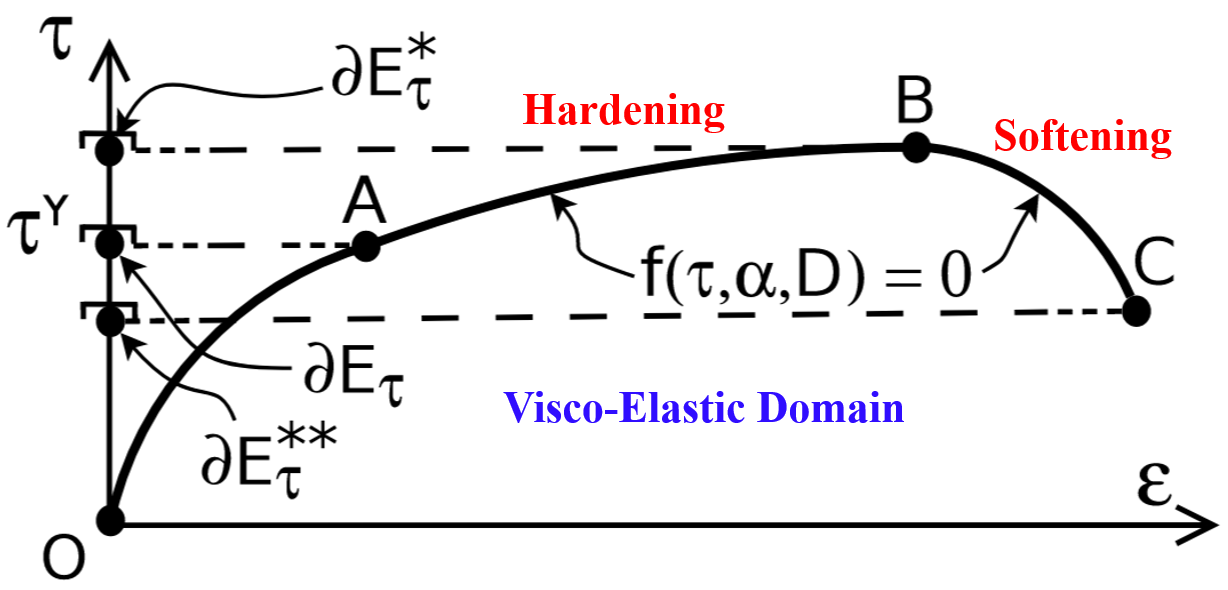}
		\caption{Stress \textit{vs.} strain response.}		
	\end{subfigure}%
	\caption{Damaged fractional visco-elasto-plastic model. (A) Constitutive 
		diagram with visco-elastic/plastic rheological elements. (B) 
		Stress response showing the yield surface expansion 
		(hardening) and contraction (softening).\label{fig:VEP_device}}
\end{figure}
Let $D(t):\mathbb{R}^+\to\Omega_D$, with $\Omega_D = [0,1)$ be a time-dependent 
and monotonically increasing internal damage variable representing the internal 
material degradation. {Our model has the following assumptions:}
\begin{assumption}
	{The visco-elastic response is linear, under an isothermal 
		strain-driven process.}	
\end{assumption}
\begin{assumption} \label{damage_driven} 
	There is a state coupling between the 
	visco-elastic {strains/har\-dening variable} 
	$\varepsilon^{ve}$, 
	$\alpha$, and damage $D$. However, the damage evolution is solely driven by 
	the visco-elastic free-energy potential.
\end{assumption}	
\begin{assumption}	\label{state} 
	{There} is no state coupling between visco-elasticity and 
	visco-plasticity.
\end{assumption}
\begin{assumption}
	The damage $D(t)$ and hardening $\alpha(t)$ are 
	{irreversible}, \textit{i.e.}, there is no material healing. 
	{Also, there are no crack closure effects.}
\end{assumption}
\begin{assumption}
	All state and internal variables are subject to homogeneous initial 
	conditions, \textit{e.g.}, $\varepsilon(0) = \varepsilon^{ve}(0) = 
	\varepsilon^{vp}(0) = \alpha(0) = D(0) = 0$.
\end{assumption}

Assumption (\ref{state}) implies {a} linearity between the 
visco-elastic and visco-plastic free-energy components, both 
{multiplicatively coupled with damage}.

\subsubsection{Free-Energy Densities}

{We} write the Helmholtz free-energy density 
$\psi(\varepsilon^{ve},\, \alpha,\,D):\mathbb{R} \times \mathbb{R}^+ \times 
\Omega_D \to \mathbb{R}^+$ for the model as:
\begin{equation}\label{eq:psi_total}
\psi(\varepsilon^{ve},\, \alpha,\,D) = \left(1 - D \right) \left( 
\bar{\psi}^{ve}(\varepsilon^{ve}) + \bar{\psi}^{vp}(\alpha) \right),
\end{equation}
where $\bar{\psi}^{ve}(\varepsilon^{ve}):\mathbb{R}\to\mathbb{R}^+$ and 
$\bar{\psi}^{vp}(\alpha):\mathbb{R}^+\to\mathbb{R}^+$ represent the undamaged 
visco-elastic and visco-plastic free-energy {densities.} 
{Utilizing} (\ref{eq:ve_potential}) for the {SB} 
elements and the {Hookean} spring, the 
{free-energy} density is given by:
\begin{align}\label{eq:vepd_potential}
	\psi(\varepsilon^{ve}, \alpha,\, D) & = \frac{1}{2} \left(1 - D\right) 
	\left[ 
	\int^\infty_0 \tilde{E}(z) \left( \int^t_0 \exp\left(-\frac{t-s}{z}\right) 
	\dot{\varepsilon}^{ve}(s)\,ds \right)^2\, dz \right. \\ \nonumber
	& \left. +  \int^\infty_0 \tilde{K}(z) \left( \int^t_0 
	\exp\left(-\frac{t-s}{z}\right) \dot{\alpha}(s)\,ds \right)^2\, dz + H 
	\alpha^2 
	\right],
\end{align}
with the following relaxation spectra for visco-elasticity and visco-plasticity:
\begin{equation*}
	\tilde{E}(z) = \frac{\mathbb{E}}{\Gamma(1-\beta_E) \Gamma(\beta_E) 
	z^{\beta_E + 
			1}}, \quad \tilde{K}(z) = \frac{\mathbb{K}}{\Gamma(1-\beta_K) 
			\Gamma(\beta_K) 
		z^{\beta_K + 1}},  
\end{equation*}
where $0 < \beta_E,\,\beta_K < 1$.

\begin{remark}[Recovery of classical free-energy potentials]
	{Similar} to the {SB} element case, we 
	recover the Hookean and Newtonian limit cases for the asymptotic values of 
	$\beta_E$, $\beta_K$. {Also, if} $D \to 0$, we recover an 
	undamaged case, and when $D \to 1$, we have $(1-D) \psi \to 0$ 
	{(material failure)}.
\end{remark}

\subsubsection{Constitutive Laws}
\label{Sec:Constitutive_laws}

{We use the Clausius-Duhem inequality 
	(\ref{eq:CD-2}) in the local form of classical thermodynamics of internal 
	variables, which induces \textit{near-equilibrium} states 
	for 
	every time $t$ of the thermodynamic process. However, the fractional 
	free-energy densities introduce memory effects and therefore 
	\textit{far-from-equili\-brium} states in the scope of rational 
	thermodynamics 
	\cite{Fabrizio2014}. Using} (\ref{eq:psi_total}) and 
(\ref{eq:kinematics}), inequality 
(\ref{eq:CD-2}) is given by:
\begin{equation}\label{eq:Clausius_Duhem}
-\rho \dot{\psi}(\varepsilon^{ve}, \alpha,D) + \tau 
\left(\dot{\varepsilon}^{ve} + \dot{\varepsilon}^{vp}\right) \ge 0,
\end{equation}
where we evaluate $\dot{\psi}$ as follows:
\begin{equation} \label{eq:psi_dot}
\dot{\psi}(\varepsilon^{ve}, \dot{\varepsilon}^{ve}, \alpha, \dot{\alpha}, D, 
\dot{D}) = \frac{\partial \psi}{\partial \varepsilon^{ve}} 
\dot{\varepsilon}^{ve} + \frac{\partial \psi}{\partial \alpha} \dot{\alpha} + 
\frac{\partial \psi}{\partial D} \dot{D}.
\end{equation}
Similar to the proof of Lemma \ref{lemma:stress_strain}, the 
partial derivatives are obtained {by} chain and Leibniz rules. 
For the first term 
on the RHS of 
(\ref{eq:psi_dot}), we have:
\begin{align*}
	\frac{\partial \psi}{\partial \varepsilon^{ve}} \dot{\varepsilon}^{ve} = & 
	\left(1 - D \right) \left[ \int^\infty_0 \tilde{E}(z) \left( \int^t_0 
	\exp\left(-\frac{t-s}{z}\right) \dot{\varepsilon}^{ve}(s)\,ds \right) dz \, 
	\dot{\varepsilon}^{ve} \right. \\
	& \left. - \int^\infty_0 \frac{\tilde{E}(z)}{z} \left( \int^t_0 
	\exp\left(-\frac{t-s}{z}\right) \dot{\varepsilon}^{ve}(s)\,ds \right)^2 dz 
	\right].
\end{align*}
Recalling (\ref{eq:dissipation_SB}), we rewrite the above equation as:
\begin{align}\label{eq:partial_ve}
	\frac{\partial \psi}{\partial \varepsilon^{ve}} \dot{\varepsilon}^{ve} = & 
	\left(1 - D \right) \left[ \int^\infty_0 \tilde{E}(z) \left( \int^t_0 
	\exp\left(-\frac{t-s}{z}\right) \dot{\varepsilon}^{ve}(s)\,ds \right) dz \, 
	\dot{\varepsilon}^{ve} - \mathcal{D}^{ve}_{mech}(\varepsilon^{ve}) \right],
\end{align}
where $\mathcal{D}^{ve}_{mech}(\varepsilon^{ve}):\mathbb{R}\to\mathbb{R}^+$ 
represents the visco-elastic mechanical energy dissipation, given 
by:
\begin{equation*}
	\mathcal{D}^{ve}_{mech}(\varepsilon^{ve}) = \int^\infty_0 
	\frac{\tilde{E}(z)}{z} \left( \int^t_0 \exp\left(-\frac{t-s}{z}\right) 
	\dot{\varepsilon}^{ve}(s)\,ds \right)^2 dz.
\end{equation*}
{Similarly}, we obtain the second term on the RHS of 
(\ref{eq:psi_dot}):
\begin{equation}\label{eq:partial_vp}
\frac{\partial \psi}{\partial \alpha} \dot{\alpha} = R(t) \dot{\alpha} 
- (1-D)\mathcal{D}^{vp}_{mech}(\alpha),
\end{equation}
where $R(t):\mathbb{R}^+\to\mathbb{R}^+$ represents the accumulated stress 
acting on the {SB} and Hooke elements on the visco-plastic part 
due to the accumulated visco-plastic strains. Recalling Lemma 
\ref{lemma:stress_strain}, $R(t)$ reads:
\begin{align*}
	R(t) & = \left(1-D\right) \left[\int^\infty_0 \tilde{K}(z) \left( \int^t_0 
	\exp\left(-\frac{t-s}{z}\right) \dot{\alpha}(s)\,ds \right) dz  + H 
	\alpha\right] \\
	& = \left(1-D\right)\left[\mathbb{K}\, {}^C_0 \mathcal{D}^{\beta_K}_t 
	\left(\alpha \right) + H\alpha\right].
\end{align*}
On the other hand, the term 
$\mathcal{D}^{vp}_{mech}(\alpha):\mathbb{R}^+\to\mathbb{R}^+$ denotes the 
visco-plastic mechanical energy dissipation in the model, which is given by:
\begin{equation*}
	\mathcal{D}^{vp}_{mech}(\alpha) = \int^\infty_0 \frac{\tilde{K}(z)}{z} 
	\left( 
	\int^t_0 \exp\left(-\frac{t-s}{z}\right) \dot{\alpha}(s)\,ds \right)^2 dz.
\end{equation*}
Finally, the direct calculation of the last term on the RHS of 
(\ref{eq:psi_dot}) yields:
\begin{equation}\label{eq:partial_D}
\frac{\partial \psi}{\partial D} \dot{D} = \left[Y^{ve}(\varepsilon^{ve}) + 
Y^{vp}(\alpha) \right] \dot{D} = Y(\varepsilon^{ve}, \alpha)\dot{D},
\end{equation}
where $Y^{ve}(\varepsilon^{ve}):\mathbb{R}\to\mathbb{R}^-$ and 
$Y^{vp}(\alpha):\mathbb{R}^+\to\mathbb{R}^-$ denote, respectively, the 
\textit{visco-elastic/plastic damage energy release rates}. 
{From} (\ref{eq:psi_total}), they are respectively given by:
\begin{equation}\label{eq:Damage_release_rate}
Y^{ve}(\varepsilon^{ve}) = -\bar{\psi}^{ve}(\varepsilon^{ve}) =  -\frac{1}{2} 
\int^\infty_0 \tilde{E}(z) \left( \int^t_0 \exp\left(-\frac{t-s}{z}\right) 
\dot{\varepsilon}^{ve}(s)\,ds \right)^2 dz.
\end{equation}
\begin{equation}\label{eq:Damage_release_rate_vp}
Y^{vp}(\alpha) = -\bar{\psi}^{vp}(\alpha) = - \frac{1}{2}\int^\infty_0 
\tilde{K}(z) \left( \int^t_0 \exp\left(-\frac{t-s}{z}\right) 
\dot{\alpha}(s)\,ds \right)^2 dz.
\end{equation}
We observe from the above result that, in principle, both visco-elastic and 
visco-plastic parts {release} bulk energy with respect to 
damage. Inserting (\ref{eq:partial_ve}), (\ref{eq:partial_vp}) and 
(\ref{eq:partial_D}) into (\ref{eq:Clausius_Duhem}), recalling Lemma 
\ref{lemma:stress_strain}, and dropping the function variables, we obtain:
\begin{align}\label{eq:CD-3}
	\left[ \tau - \left(1 - D \right) \mathbb{E}\, {}^C_0 
	\mathcal{D}^{\beta_E}_t 
	\left(\varepsilon^{ve}\right) \right] \dot{\varepsilon}^{ve} & + \tau 
	\dot{\varepsilon}^{vp} - R \dot{\alpha}  \\
	& - Y \dot{D} \nonumber + (1-D)\left(\mathcal{D}^{ve}_{mech} + 
	\mathcal{D}^{vp}_{mech}\right) \ge 0.
\end{align}
%
%
Since the strain rate $\dot{\varepsilon}^{ve}$ in (\ref{eq:CD-3}) is arbitrary, 
without {violating} the inequality, we can set its 
multiplying argument to zero, and obtain the following stress-strain 
relationship:
\begin{equation}\label{eq:stress_strain_1}
\tau(t) = \left(1 - D \right)\mathbb{E}\, {}^C_0 \mathcal{D}^{\beta_E}_t 
\left(\varepsilon^{ve}\right),
\end{equation}
and alternatively, {using} (\ref{eq:kinematics}), we obtain:
\begin{equation}\label{eq:stress_strain_2}
\tau(t) = \left(1 - D \right)\mathbb{E}\, {}^C_0 \mathcal{D}^{\beta_E}_t 
\left(\varepsilon - \varepsilon^{vp} \right),
\end{equation}
and hence, the total energy dissipation (\ref{eq:CD-3}) becomes:
\begin{equation}\label{eq:DMech}
\tau \dot{\varepsilon}^{vp} - R \dot{\alpha} - Y \dot{D} + 
(1-D)\left(\mathcal{D}^{ve}_{mech} + \mathcal{D}^{vp}_{mech}\right) \ge 0.
\end{equation}
{Hence, we obtained the stress-strain relationships and dissipation potentials.}

\subsubsection{Evolution Laws for Visco-Plasticity and Damage}
\label{Sec:VP_D}

In order to obtain the kinematic equations for the internal variables, we 
{define} a combined hardening and damage dissipation potential 
$F\left(\tau, \alpha, Y, D 
\right):\mathbb{R}\times\mathbb{R}^+\times\mathbb{R}^-\times\mathbb{R}^+ \to 
\mathbb{R}$, in the form \cite{Lemaitre1996, Lemaitre2005}:
\begin{equation}\label{eq:dissipation_potential}
F\left(\tau, \alpha, Y, D \right) := f\left( \tau, \alpha, D \right) + F_D 
\left( Y^{ve}, D \right),
\end{equation}
where $f\left( \tau, \alpha, D 
\right):\mathbb{R}\times\mathbb{R}^+\times\mathbb{R}^+\to\mathbb{R}^-\cup\lbrace
0\rbrace$ represents a yield function, defined here as the difference 
between the absolute value of the applied stress in the device and the stress 
acting on the visco-plastic part \cite{Suzuki2016}:
\begin{align}\label{eq:Yield_function}
	f\left(\tau, \alpha, D \right) & := \lvert \tau \rvert - \left[ (1-D)\tau^Y 
	+ R 
	\right] \nonumber \\
	& = \lvert \tau \rvert - \left(1-D\right)\left[ \tau^Y +  \mathbb{K}\, 
	{}^C_0 
	\mathcal{D}^{\beta_K}_t \left(\alpha \right) + H\alpha \right],
\end{align}
which softens the visco-plastic stresses. 
{
	\begin{lemma}
		The set of admissible stresses lies in a closed convex 
		space (\textit{see Fig.\ref{fig:VEP_device}}) with respect to the 
		associated thermodynamic variables $\tau$ and $R$ \cite{Lemaitre2005}, 
		given by:
		\begin{equation}\label{eq:admissible_stress}
		E_\tau = \lbrace \tau \in \mathbb{R} \vert f(\tau,\alpha,D) < 0 \rbrace.
		\end{equation}
		The boundary of $E_\tau$, denoted by $\partial E_\tau$, is the 
		convex set given by:
		\begin{equation*}
			\partial E_\tau = \lbrace \tau \in \mathbb{R} \vert 
			f(\tau,\alpha,D) = 0 
			\rbrace,
		\end{equation*}
		where $f(\tau,\alpha,D) = 0$ denotes the \textit{yield condition} in 
		classical plasticity.
	\end{lemma}
	\begin{proof}
		See Appendix \ref{Ap:Conxevity}.
\end{proof}}

The term $F_D \left( Y^{ve}, D 
\right):\mathbb{R}^-\times\mathbb{R}^+\to\mathbb{R}^+$ represents a damage 
potential driven by the plastic strains and visco-elastic free-energy 
(\textit{see Assumption \ref{damage_driven}}), {where we adopt 
	Lemaitre's form 
	for ductile materials \cite{Lemaitre2005}:}
\begin{equation}\label{eq:damage_potential}
F_D(Y,D) := \frac{S}{(s+1)(1-D)}\left(- \frac{Y^{ve}}{S} \right)^{s+1},
\end{equation}
{where $S \in \mathbb{R}^+$ $[Pa]$ and $s \in \mathbb{R}^+$ 
	represent 
	material parameters,  identified, \textit{e.g.}, by Cao \textit{et al.} 
	\cite{Cao2013} for a Zirconium alloy, and 
	by Bouchard \textit{et al.} \cite{Bouchard2011} for highly ductile metals. 
	In 
	the latter, an inverse power-law form for $F_D$ was defined with respect to 
	the 
	equivalent plastic strains to avoid damage over-estimation. The} 
sensitivity of {Lemaitre's} model with 
respect to {$S$ and $s$} was studied by Roux and Bouchard 
\cite{Roux2015}. 

{From the defined yield function} (\ref{eq:Yield_function}), and 
the 
principle of maximum plastic dissipation \cite{Simo1998}, the following 
{properties hold:} \textbf{i)} associativity of the flow rule, 
\textbf{ii)} associativity in the 
hardening law, \textbf{iii)} Kuhn-Tucker complimentary conditions, and 
\textbf{iv)} convexity of $E_\tau$. Therefore, we obtain a set of evolution 
equations for $\varepsilon^{vp}$, $\alpha$ and $D$:
\begin{equation*}
	\dot{\varepsilon}^{vp} = \frac{\partial f}{\partial \tau} \dot{\gamma}, 
	\quad 
	\dot{\alpha} = -\frac{\partial f}{\partial R} \dot{\gamma}, \quad \dot{D} = 
	-\frac{\partial F_D}{\partial Y^{ve}} \dot{\gamma},
\end{equation*}
where $\dot{\gamma}(t):\mathbb{R}^+\to\mathbb{R}^+$ denotes the plastic slip 
rate. {For simplicity, we consider only variations of the 
	potential $F_D$ with respect to the free-energy from the visco-elastic 
	component for the damage evolution}. Evaluating the 
above equations using (\ref{eq:Yield_function}) and 
(\ref{eq:damage_potential}), we obtain, respectively, the evolution for 
visco-plastic strains, hardening variable, and damage:
\begin{equation}\label{eq:evol_vp_damage}
\dot{\varepsilon}^{vp} = \mathrm{sign}(\tau) \dot{\gamma},
\end{equation}
\begin{equation}\label{eq:evol_hardening_damage}
\dot{\alpha} = \dot{\gamma},
\end{equation}
\begin{equation}\label{eq:evol_damage}
\dot{D} = \frac{\dot{\gamma}}{(1-D)}\left(- \frac{Y^{ve}}{S} \right)^s,
\end{equation}
where the first two evolution laws coincide with the ones defined for the model 
M1 by Suzuki \textit{et al.} \cite{Suzuki2016} for fractional 
visco-elasto-plasticity. 
{
	\begin{remark}
		The obtained nonlinear damage evolution (\ref{eq:evol_damage})
		coincides with local Lemaitre's form \cite{Lemaitre1996, Lemaitre2005}. 
		However, due to the time-fractional form of $Y^{ve}$, power-law memory 
		effects 
		for damage are introduced in the model.
\end{remark}}

\begin{theorem}[Positive dissipation]\label{thm:positive_D}
	The mechanical dissipation for the damaged, fractional 
	visco-elasto-plastic model is positive and given by,
	\begin{equation*}
		\left(1-D(t)\right) 
		\left[\tau^Y \dot{\gamma}(t) + 
		\mathcal{D}^{ve}_{mech}(\varepsilon^{ve}) + 
		\mathcal{D}^{vp}_{mech}(\alpha)\right] - Y(\varepsilon^{ve},\alpha) 
		\dot{D}(t) \ge \, 0,
	\end{equation*}
	where the above Clausius-Duhem inequality holds. Therefore, the defined 
	Helmholtz free-energy density (\ref{eq:vepd_potential}), the obtained 
	stress-strain relationship (\ref{eq:stress_strain_2}) and evolution 
	equations (\ref{eq:evol_vp_damage})-(\ref{eq:evol_damage}) of the developed 
	model are thermodynamically {admissible}.
\end{theorem}
{
	\begin{proof}
		See Appendix \ref{Ap:positive_D}.
\end{proof}}

\subsection{Time-Fractional Integration}

We develop two new algorithms for time-fractional integration 
of the developed {model. The first one is a 
	semi-implicit \textit{fractional return-mapping algorithm}}, 
that can be implemented in zero- or one- dimensional {systems} 
as a constitutive box. {The second one is an FD discretization 
	for the fractional Helmholtz free-energy density and} damage energy release 
rate $Y$ (\ref{eq:Damage_release_rate}). Let $t \in (0, T]$, 
{and an uniform time grid given by $t_n = n \Delta t$, with $n = 
	0,\,1,\,\dots,\,N$ and time-step size $\Delta t = T/N$.}

\subsubsection{Semi-Implicit Fractional Return-Mapping Algorithm}

We employ a backward-Euler scheme {considering} all variables to be 
implicit, except the damage $D$ in the stress-strain relationship and yield 
function. We refer the readers to \cite{Bouchard2011} for a comparison between 
implicit/semi-implicit integer-order return-mapping algorithms. Such explicit 
treatment of $D$ 
{weakly couples} the damage and plastic slip, simplifying the 
visco-plastic time-integration. Given known total strains $\varepsilon_n$ at 
time $t_n$, and a strain increment $\Delta \varepsilon_{n+1}$ we 
{have $ 
	\varepsilon_{n+1} = \varepsilon_{n} + \Delta \varepsilon_{n+1}$}. The 
	discrete 
form of the stress-strain relationship (\ref{eq:stress_strain_2}) reads:
\begin{equation}\label{eq:discretized_stress}
\tau_{n+1} = (1-D_n)\mathbb{E}\, {}^C_0 \mathcal{D}^{\beta_E}_t 
\left(\varepsilon - \varepsilon^{vp} \right)\big|_{t=t_{n+1}}.
\end{equation}
The backward-Euler discretization of the flow rule (\ref{eq:evol_vp_damage}) 
yields:
\begin{equation}\label{eq:discretized_vp}
\varepsilon^{vp}_{n+1} = \varepsilon^{vp}_{n} + \sign(\tau_{n+1})\Delta\gamma,
\end{equation}
with $\Delta\gamma = \gamma_{n+1} - \gamma_n$ representing the plastic slip 
increment in the interval $[t_n,\,t_{n+1}]$. Similarly, the discretization of 
the hardening law (\ref{eq:evol_hardening_damage})  and the damage evolution 
(\ref{eq:evol_damage}) are given,  respectively, by
\begin{equation}\label{eq:discretized_alpha}
\alpha_{n+1} = \alpha_n + \Delta\gamma,
\end{equation}
\begin{equation}\label{eq:discretized_damage}
D_{n+1} = D_n + 
\frac{\Delta\gamma}{1-D_{n+1}}\left(\frac{Y^{ve}_{n+1}}{S}\right)^s,
\end{equation}
with the following discrete form for the damage energy release rate 
(\ref{eq:Damage_release_rate}):
\begin{equation*}
	Y^{ve}_{n+1} = -\bar{\psi}^{ve}_{n+1} = - \frac{1}{2}\int^\infty_0 
	\tilde{E}(z) 
	\left( \int^{t_{n+1}}_0 \exp\left(-\frac{t_{n+1}-s}{z}\right) 
	\dot{\varepsilon}^{ve}(s)\,ds \right)^2 dz.
\end{equation*}
Similarly, the yield function (\ref{eq:Yield_function}) evaluated at $t_{n+1}$ 
is given by:
\begin{equation}\label{eq:discrete_f}
f_{n+1} = \lvert\tau_{n+1}\rvert - \left(1-D_n\right)\left[\tau^Y + 
\mathbb{K}\, {}^C_0 \mathcal{D}^{\beta_K}_t (\alpha)\big|_{t=t_{n+1}} + 
H\alpha_{n+1}\right].
\end{equation}
%

{We utilize trial states, were we freeze the internal variables 
	(except for damage) for the prediction step at $t_{n+1}$.} Therefore, we 
	have:
\begin{equation*}
	\varepsilon^{vp^{trial}}_{n+1} = \varepsilon^{vp}_n, \quad 
	\alpha^{trial}_{n+1} 
	= \alpha_n.
\end{equation*}
{The trial visco-elastic stress and yield function are 
	given, respectively, by}
\begin{equation}\label{eq:trial_stress}
\tau^{trial}_{n+1} = (1-D_n)\mathbb{E}\, {}^C_0 \mathcal{D}^{\beta_E}_t 
(\varepsilon - \varepsilon^{vp^{trial}} )\big|_{t=t_{n+1}},
\end{equation}
\begin{equation*}
	f^{trial}_{n+1} = \lvert\tau^{trial}_{n+1}\rvert - 
	\left(1-D_n\right)\left[\tau^Y + \mathbb{K}\, {}^C_0 
	\mathcal{D}^{\beta_K}_t 
	(\alpha^{trial})\big|_{t=t_{n+1}} + H\alpha^{trial}\right].
\end{equation*}
Substituting (\ref{eq:discretized_vp}) into (\ref{eq:discretized_stress}) and 
recalling (\ref{eq:trial_stress}), we obtain:
\begin{equation*}
	\tau_{n+1} = \tau^{trial}_{n+1} - \sign(\tau_{n+1})(1-D_n)\mathbb{E}\, 
	{}^C_0 
	\mathcal{D}^{\beta_E}_t \left(\Delta\gamma \right)\big|_{t=t_{n+1}},
\end{equation*}
where we observe that
\begin{equation*}
	\left[\lvert\tau_{n+1}\rvert + (1-D_n)\mathbb{E}\, {}^C_0 
	\mathcal{D}^{\beta_E}_t \left(\Delta\gamma 
	\right)\big|_{t=t_{n+1}}\right]\sign(\tau_{n+1}) = 
	\lvert\tau^{trial}_{n+1}\rvert\sign(\tau^{trial}_{n+1}).
\end{equation*}
{Since the argument inside brackets on the LHS above is 
	positive, we note} that $\sign(\tau_{n+1}) = \sign(\tau^{trial}_{n+1})$. 
Hence, we have the updated stress:
\begin{equation}\label{eq:tau_tautrial}
\tau_{n+1} = \tau^{trial}_{n+1} - \sign(\tau^{trial}_{n+1})(1-D_n)\mathbb{E}\, 
{}^C_0 \mathcal{D}^{\beta_E}_t \left(\Delta\gamma \right)\big|_{t=t_{n+1}}.
\end{equation}
Our last step is to derive the closure to for the plastic slip $\Delta\gamma$. 
Substituting (\ref{eq:tau_tautrial}) and (\ref{eq:discretized_alpha}) into 
(\ref{eq:discrete_f}), we obtain:
{\small \begin{equation*}
		f_{n+1} = f^{trial}_{n+1} - (1-D_n)\left[\mathbb{E}\, {}^C_0 
		\mathcal{D}^{\beta_E}_t \left(\Delta\gamma \right)\big|_{t=t_{n+1}} - 
		\mathbb{K}\, {}^C_0 \mathcal{D}^{\beta_K}_t \left(\Delta\gamma 
		\right)\big|_{t=t_{n+1}} - H\Delta\gamma\right].
\end{equation*} }
Finally, setting the discrete yield condition $f_{n+1} = 0$, we obtain the 
following multi-term fractional differential equation for the plastic slip:
\begin{equation}\label{eq:multi_term_plastic_slip}
\boxed{\mathbb{E}\, {}^C_0 \mathcal{D}^{\beta_E}_t \left(\Delta\gamma 
	\right)\big|_{t=t_{n+1}} + \mathbb{K}\, {}^C_0 \mathcal{D}^{\beta_K}_t 
	\left(\Delta\gamma \right)\big|_{t=t_{n+1}} + H\Delta\gamma = 
	\frac{f^{trial}_{n+1}}{(1-D_n)}.}
\end{equation}
After solving (\ref{eq:multi_term_plastic_slip}) for 
{$\Delta\gamma$, we directly update the internal variables 
	$\alpha_{n+1}$ and $\varepsilon^{vp}_{n+1}$. The damage update is done 
	through Newton iteration. Let $P^{k}_{n+1}$ given at a sub-iteration $k$:}
\begin{equation*}
	P^{k}_{n+1} = D^k_{n+1} - D_n - 
	\frac{\Delta\gamma}{1-D^k_{n+1}}\left(\frac{Y^{ve}_{n+1}}{S}\right)^s,
\end{equation*}
with the following derivative, obtained analytically:
\begin{equation*}
	\frac{dP}{dD^k}\big|_{t=t_{n+1}} = 1 + 
	\frac{\Delta\gamma}{(1-D^k_{n+1})^2}\left(\frac{Y^{ve}_{n+1}}{S}\right)^s.
\end{equation*}
Therefore, the new iterated damage is given by:
\begin{equation*}
	D^{k+1}_{n+1} = D^k_{n+1} - 
	\frac{P^{k}_{n+1}}{\left(dP/dD^k\right)\big|_{t=t_{n+1}}}.
\end{equation*}

The developed fractional return-mapping algorithm is 
{summarized} in Algorithm \ref{alg:return-mapping}. 

\begin{algorithm}[t!]
	\caption{Fractional return-mapping algorithm.}
	\label{alg:return-mapping}
	\begin{algorithmic}[1]
		\STATE{Database for $\varepsilon$, $\varepsilon^{vp}$, $\alpha$, 
			$\Delta\gamma$, $D_n$ and total strain $\varepsilon_{n+1}$.}
		\STATE{$\varepsilon^{vp^{trial}}_{n+1} = \varepsilon^{vp}_n, \quad 
			\alpha^{trial}_{n+1} = \alpha_n$}
		\STATE{$\tau^{trial}_{n+1} = (1-D_n)\mathbb{E}\, {}^C_0 
			\mathcal{D}^{\beta_E}_t (\varepsilon - \varepsilon^{vp^{trial}} 
			)\big|_{t=t_{n+1}}$}
		\STATE{$f^{trial}_{n+1} = \lvert\tau^{trial}_{n+1}\rvert - 
			\left(1-D_n\right)\left[\tau^Y + \mathbb{K}\, {}^C_0 
			\mathcal{D}^{\beta_K}_t (\alpha^{trial})\big|_{t=t_{n+1}} + 
			H\alpha^{trial}\right]$}
		\IF{$f^{trial}_{n+1} \le 0$}
		\STATE{$\varepsilon^{vp}_{n+1} = \varepsilon^{vp}_n$,\quad  
			$\alpha_{n+1} = \alpha_n$,\quad $D_{n+1} = D_n$,\quad $\tau_{n+1} = 
			\tau^{trial}_{n+1}$.}
		\ELSE
		\STATE{Solve for $\Delta\gamma$:}
		\STATE{$\mathbb{E}\, {}^C_0 \mathcal{D}^{\beta_E}_t \left(\Delta\gamma 
			\right)\big|_{t=t_{n+1}} + \mathbb{K}\, {}^C_0 
			\mathcal{D}^{\beta_K}_t 
			\left(\Delta\gamma \right)\big|_{t=t_{n+1}} + H\Delta\gamma = 
			f^{trial}_{n+1}/(1-D_n)$}
		\STATE{$\tau_{n+1} = \tau^{trial}_{n+1} - 
			\sign(\tau^{trial}_{n+1})(1-D_n)\mathbb{E}\, {}^C_0 
			\mathcal{D}^{\beta_E}_t \left(\Delta\gamma 
			\right)\big|_{t=t_{n+1}}$}
		\STATE{$\varepsilon^{vp}_{n+1} = \varepsilon^{vp}_{n} + 
			\sign(\tau_{n+1})\Delta\gamma$}
		\STATE{$\alpha_{n+1} = \alpha_n + \Delta\gamma$}
		\STATE{{\small$Y^{ve}_{n+1} = - \frac{1}{2}\int^\infty_0 \tilde{E}(z) 
				\left( 
				\int^{t_{n+1}}_0 \exp\left(-\frac{t_{n+1}-s}{z}\right) 
				\dot{\varepsilon}^{ve}(s)\,ds \right)^2 dz$} (Algorithm 
			\ref{alg:1}).}
		\WHILE{$\vert P^{k}_{n+1}\vert > \epsilon$}
		\STATE{$P^{k}_{n+1} = D^k_{n+1} - D_n - 
			\frac{\Delta\gamma}{1-D^k_{n+1}}\left(\frac{Y^{ve}_{n+1}}{S}\right)^s$}
		\STATE{$\left(dP/dD^k\right)\big|_{t=t_{n+1}} = 1 + 
			\frac{\Delta\gamma}{(1-D^k_{n+1})^2}\left(\frac{Y^{ve}_{n+1}}{S}\right)^s$}
		\STATE{$D^{k+1}_{n+1} = D^k_{n+1} - 
			\frac{P^{k}_{n+1}}{\left(dP/dD^k\right)\big|_{t=t_{n+1}}}$}
		\ENDWHILE		
		\ENDIF
		
	\end{algorithmic}
\end{algorithm}

\subsubsection{Numerical Discretization of Fractional Operators}
\label{Sec:FD}

The fractional derivatives in the 
fractional return-mapping Algorithm \ref{alg:return-mapping} are evaluated 
implicitly using the L1 FD method \cite{Lin2007}. 
{Let $u(t):\mathbb{R}^+\to\mathbb{R}$. The time-fractional 
	Caputo derivative of order $0 < \beta < 1$ is discretized as:}
\begin{equation}
\label{eq:FDM}
{}^C_0{\mathcal{D}}_t^\beta u(t) \Big|_{t=t_{n+1}} = 
\frac{1}{\Gamma(2-\beta)} \sum_{j=0}^{n} d_j \frac{u_{n+1-j}-u_{n-j} 
}{\Delta t^\beta} + r^{n+1}_{\Delta t},
\end{equation}
where $r^{n+1}_{\Delta t} \le C_u \Delta t^{2-\beta}$ and $d_j := 
(j+1)^{1-\beta} 
- j^{1-\beta}, j=0,1,\dots,n$. The above expression can be rewritten and 
approximated as:
\begin{equation*}
	{}^C_0{\mathcal{D}}_t^\beta   u(t)\Big|_{t=t_{n+1}} \approx 
	\frac{1}{\Delta t^\beta \Gamma(2-\beta)} \left[ u_{n+1} - u_n + 
	\mathcal{H}^{\beta}u \right],
\end{equation*}
where the so-called \textit{history term} $\mathcal{H}^{\beta}u$ is given by:
\begin{equation} \label{Eq:History}
\mathcal{H}^{\beta}u = \sum_{j=1}^{n} d_j \left[ u_{n+1-j}-u_{n-j} \right].
\end{equation}
{Using (\ref{eq:FDM}) does not cause any loss of accuracy for 
	the return-mapping, since the backward-Euler approach for internal 
	variables is 
	first-order accurate. For trial state variables $u^{trial}_{n+1} = u_n$, 
	the discretized Caputo fractional derivatives are given by:}
\begin{equation} \label{Eq:Caputo_trial}
{}^C_0{\mathcal{D}}_t^\beta  u^{trial}(t)\Big|_{t=t_{n+1}} 
\approx \frac{\mathcal{H}^{\beta}u}{\Delta t^\beta \Gamma(2-\beta)}.
\end{equation}

\textbf{Free-Energy/Damage Energy Release Rate:} {We now 
	discretize the 
	visco-elastic damage energy release rate $Y^{ve} = 
	-\bar{\psi}^{ve}$. We first rewrite (\ref{eq:ve_potential}) as 
	\cite{Lion1997}:}
\begin{equation}\label{eq:free_energy_double}
\psi(\varepsilon) = \frac{\mathbb{E}}{2 \Gamma(1-\beta)} \int^t_0 \int^t_0 
\frac{\dot{\varepsilon}(s_1)\dot{\varepsilon}(s_2)}{\left( 2t -s_1 - s_2 
	\right)^\beta}\,ds_1\,ds_2.
\end{equation}
{We then decompose} the integral signs of 
(\ref{eq:free_energy_double}) 
into a discrete summation of $n$ integrals and approximate 
$\dot{\varepsilon}(t)$ using a backward-Euler scheme to obtain,
\begin{align}\label{eq:L1_potential_1}
	& \psi(\varepsilon_{n+1}) = \frac{\mathbb{E}}{2 \Gamma(1-\beta)} 
	\int^{t_{n+1}}_0 \int^{t_{n+1}}_0 
	\frac{\dot{\varepsilon}(s_1)\dot{\varepsilon}(s_2)}{\left( 2t_{n+1} -s_1 - 
	s_2 
		\right)^\beta}\,ds_1\,ds_2 \nonumber \\
	& = \frac{\mathbb{E}}{2 \Gamma(1-\beta)} \sum^n_{i=0} \sum^n_{j=0} 
	\int^{t_{i+1}}_{t_i} \int^{t_{j+1}}_{t_j} \frac{\Delta\varepsilon_{i+1} 
		\Delta\varepsilon_{j+1} }{\Delta t^2 \left( 2 
		t_{n+1} - s_1 - s_2 \right)^\beta}\,ds_1\,ds_2 + 
		\tilde{r}^{n+1}_{\Delta t},
\end{align}
with $\Delta \varepsilon_{k+1} = \varepsilon_{k+1} - \varepsilon_k$
{
	\begin{theorem}
		The local truncation error $\tilde{r}^{n+1}_{\Delta t}$ for 
		(\ref{eq:L1_potential_1}) satisfies
		\begin{equation} \label{eq:r_bound}
		\tilde{r}^{n+1}_{\Delta t} \le C \Delta t^{2-\beta},
		\end{equation}
		where $C$ denotes a constant depending only on the strain 
		$\varepsilon(t)$. 
	\end{theorem}
	\begin{proof}
		See Appendix \ref{Ap:Truncation}.
\end{proof}}

{Let the first term of the RHS of (\ref{eq:L1_potential_1}) be 
	the approximation $\psi^\delta_{n+1} \approx 
	\psi(\varepsilon_{n+1})$ evaluated at $t=t_{n+1}$}. Performing 
a change of variables $v_1 = t_{n+1} - s_1$ and $v_2 = t_{n+1} - s_2$, we 
obtain:
\begin{equation}\label{eq:L1_potential_2}
\psi^\delta_{n+1} = \mathbb{E}^* \sum^n_{i=0} \sum^n_{j=0} 
\frac{\Delta \varepsilon_{i+1} \Delta\varepsilon_{j+1}}{\Delta t^2} 
\int^{t_{n+1-i}}_{t_{n-i}} 
\int^{t_{n+1-j}}_{t_{n-j}} \left( v_1 + v_2 \right)^{-\beta} \,dv_1\,dv_2,
\end{equation}
with $\mathbb{E}^* = \mathbb{E}/\left(2 \Gamma(1-\beta)\right)$. Using the 
symmetry between the indices of strains and integration limits in 
(\ref{eq:L1_potential_2}), we obtain:
\begin{equation}\label{eq:L1_potential_3}
\psi^\delta_{n+1} = \mathbb{E}^* \sum^n_{i=0} \sum^n_{j=0} 
\frac{\Delta\varepsilon_{n-i+1} \Delta \varepsilon_{n-j+1}}{\Delta t^2} 
\int^{t_{i+1}}_{t_{i}} 
\int^{t_{j+1}}_{t_{j}} \left( v_1 + v_2 \right)^{-\beta} dv_1 dv_2.
\end{equation}
We can analytically evaluate the double integral sign in 
(\ref{eq:L1_potential_3}) to obtain:
\begin{align}\label{eq:L1_potential_4}
	\int^{t_{i+1}}_{t_{i}} & \int^{t_{j+1}}_{t_{j}} \left( v_1 + v_2 
	\right)^{-\beta} \,dv_1\,dv_2 = \nonumber \\
	& \frac{\Delta t^{2-\beta}}{(1-\beta) (2-\beta)} \left[ (i+j)^{2-\beta} - 
	2(i + 
	j + 1)^{2-\beta} + (i + j + 2)^{2-\beta} \right].
\end{align}
Substituting (\ref{eq:L1_potential_4}) into (\ref{eq:L1_potential_3}), 
{we 
	obtain the discrete free-energy density,}
\begin{equation}\label{eq:final_index}
\psi^\delta_{n+1} = \frac{\mathbb{E}}{2 \Delta t^\beta \Gamma(3-\beta)} 
\sum^n_{i=0} \sum^n_{j=0} b^{(\beta)}_{ij} 
(\varepsilon_{n+1-i}-\varepsilon_{n-i}) 
(\varepsilon_{n+1-j}-\varepsilon_{n-j}), 
\end{equation}
with the following entries for the convolution weight matrix:
\begin{equation}
b^{(\beta)}_{ij} = (i+j)^{2-\beta} - 2(i + j + 1)^{2-\beta} + (i + j + 
2)^{2-\beta}, \quad i,\,j = 0,\,1,\,\dots,\,n. \nonumber
\end{equation}
We can also rewrite (\ref{eq:final_index}) as the following matrix-vector 
product:
\begin{equation}\label{eq:final_matvec}
\psi^\delta_{n+1} =  \frac{\mathbb{E}}{2 \Delta t^\beta \Gamma(3-\beta)} 
\boldsymbol{\Delta \varepsilon}^T_{n+1} \mathbf{B}_{n+1} \boldsymbol{\Delta 
	\varepsilon}_{n+1},
\end{equation}
where we note that $\mathbf{B}_{n+1}$ is an $n \times n$ Hankel matrix of 
convolution weights {with} $2n-1$ unique 
entries $b^{(\beta)}_{ij}$. The $n\times 1$ vector $\boldsymbol{\Delta 
	\varepsilon}_{n+1}$ is given by:
\begin{equation}\label{eq:DE}
\boldsymbol{\Delta \varepsilon}_{n+1} = \left[ \varepsilon_{n+1} - 
\varepsilon_n,\,\, \varepsilon_{n} - \varepsilon_{n-1},\,\, \dots,\,\, 
\varepsilon_2 - \varepsilon_1,\,\, \varepsilon_1 - \varepsilon_0 \right]^T.
\end{equation}

\textbf{Fast Computation of Matrix-Vector Products:} The form 
(\ref{eq:final_matvec}) requires a {full} matrix-vector product 
with complexity $\mathcal{O}(n^2)$ {for} every time-step, 
{and consequently $\mathcal{O}(N^3)$ for full 
	time-integration}. {Our aim is to reduce such complexity by 
	leveraging the obtained matrix forms. Since $\mathbf{B}$ is a Hankel 
	matrix, it relates to a Toeplitz matrix} $\mathbf{T}_{n+1}$ through 
$\mathbf{B}_{n+1} = \mathbf{T}_{n+1}\mathbf{J}_{n+1}$, where $\mathbf{J}_{n+1}$ 
represents a {reflection} matrix with ones in the secondary 
diagonal and zero everywhere else. Therefore, we obtain:
\begin{equation}\label{eq:slow}
\psi^\delta_{n+1} =  \frac{\mathbb{E}}{2 \Delta t^\beta
	\Gamma(3-\beta)} \boldsymbol{\Delta \varepsilon}^T_{n+1} \mathbf{T}_{n+1} 
\mathbf{J}_{n+1} 
\boldsymbol{\Delta \varepsilon}_{n+1}.
\end{equation}
{The Toeplitz matrix has a circulant embedding of size $2n\times 
	2n$ 
	\cite{Feng2007},} fully described by a $2n\times 1$ vector of unique 
coefficients:
\begin{equation}\label{eq:vec_B}
\mathbf{c}^{(\beta)}_{n+1} = \left[ b^{(\beta)}_{0,n},\,\, 
b^{(\beta)}_{1,n},\,\, \dots,\,\, b^{(\beta)}_{n,n},\, 0,\,\, 
b^{(\beta)}_{0,0},\,\, b^{(\beta)}_{0,1},\,\, \dots,\,\, b^{(\beta)}_{0,n-1} 
\right]^T.
\end{equation}
{Let the following} zero-padded vector $\boldsymbol{\Delta 
	\varepsilon}^*_{n+1}$, with size $2n \times 1$:
\begin{equation}\label{eq:vec_epstar}
\boldsymbol{\Delta \varepsilon}^*_{n+1} = \left[ (\boldsymbol{\Delta 
	\varepsilon}^f_{n+1})_{n \times 1},\,\,\, (\mathbf{0})_{n\times 1} 
	\right]^T,
\end{equation}
where $\boldsymbol{\Delta \varepsilon}^f_{n+1} = 
\mathbf{J}_{n+1}\boldsymbol{\Delta \varepsilon}_{n+1}$ {denotes 
	the reflection} of $\boldsymbol{\Delta \varepsilon}_{n+1}$, given by:
\begin{equation}\label{eq:vec_epflip}
\boldsymbol{\Delta \varepsilon}^f_{n+1} = \left[ \varepsilon_1 - 
\varepsilon_0,\,\,\varepsilon_2 - \varepsilon_1,\,\, \dots,\,\,\varepsilon_{n} 
- \varepsilon_{n-1},\,\,\varepsilon_{n+1} - \varepsilon_n \right]^T.
\end{equation}
Finally, {we obtain the fast form of} (\ref{eq:final_matvec}) 
for every time-step $t_{n+1}$:
\begin{equation}\label{eq:fast}
\psi^\delta_{n+1} =  \frac{\mathbb{E}}{2 \Delta t^\beta \Gamma(3-\beta)} 
\boldsymbol{\Delta \varepsilon}^T_{n+1} \mathcal{F}^{-1}\left( 
\mathcal{F}(\mathbf{c}^{(\beta)}_{n+1}) \odot \mathcal{F} (\boldsymbol{\Delta 
	\varepsilon}^*_{n+1}) \right),
\end{equation} 
where $\mathcal{F}(\cdot)$ and $\mathcal{F}^{-1}(\cdot)$ denote, respectively, 
the forward and inverse FFTs and $\odot$ represents the Hadamard entry-wise 
product. {Recalling} $Y^{ve}(\varepsilon^{ve}) = 
-\bar{\psi}^{ve}(\varepsilon^{ve})$, the discrete 
damage energy release rate is given by:
\begin{equation}\label{eq:fastY}
\boxed{Y^{ve}_{n+1} =  -\frac{\mathbb{E}}{2 \Delta t^{\beta_E} 
		\Gamma(3-\beta_E)} \boldsymbol{\Delta \varepsilon}^{ve^T}_{n+1} 
	\mathcal{F}^{-1}\left( \mathcal{F}(\mathbf{c}^{(\beta_E)}_{n+1}) \odot 
	\mathcal{F} (\boldsymbol{\Delta \varepsilon}^{ve^*}_{n+1}) \right),}
\end{equation} 
where,
\begin{equation}\label{eq:DE_ve}
\boldsymbol{\Delta \varepsilon}^{ve}_{n+1} = \left[ \varepsilon^{ve}_{n+1} - 
\varepsilon^{ve}_n,\,\, \varepsilon^{ve}_{n} - \varepsilon^{ve}_{n-1},\,\, 
\dots,\,\, \varepsilon^{ve}_2 - \varepsilon^{ve}_1,\,\, \varepsilon^{ve}_1 - 
\varepsilon^{ve}_0 \right]^T,
\end{equation}
and with $\boldsymbol{\Delta \varepsilon}^{ve^*}_{n+1}$ {being 
	the reflected} and zero-padded form of (\ref{eq:DE_ve}). {Also}, 
the vector 
$\mathbf{c}^{(\beta_E)}_{n+1}$ is 
given by:
\begin{equation}\label{eq:vec_B_ve}
\mathbf{c}^{(\beta_E)}_{n+1} = \left[ b^{(\beta_E)}_{0,n},\,\, 
b^{(\beta_E)}_{1,n},\,\, \dots,\,\, b^{(\beta_E)}_{n,n},\, 0,\,\, 
b^{(\beta_E)}_{0,0},\,\, b^{(\beta_E)}_{0,1},\,\, \dots,\,\, 
b^{(\beta_E)}_{0,n-1} \right]^T,
\end{equation}
with $b^{(\beta_E)}_{ij} = (i+j)^{2-\beta_E} - 2(i + j + 1)^{2-\beta_E} + (i + 
j + 2)^{2-\beta_E}$ and $i,\,j = 0,\,1,\,\dots,\,n$. Algorithm \ref{alg:1} 
{demonstrates the numerical evaluation of the damage energy 
	release rate for every} time-step $t = t_{n+1}$.

\begin{algorithm}[t!]
	\caption{Fast computation of fractional damage energy release rate.}
	\label{alg:1}
	\begin{algorithmic}[1]
		\STATE{Database: $\boldsymbol{\varepsilon}^{ve}$ and $2N-1$ 
			coefficients 
			$b^{(\beta_E)}_{0,0},\,\dots,\,b^{(\beta_E)}_{0,N},\,b^{(\beta_E)}_{1,N},\,\dots,\,b^{(\beta_E)}_{N,N}$.}
		\STATE{Compute $\boldsymbol{\Delta \varepsilon}^{ve}_{n+1}$ using 
			(\ref{eq:DE}), and 
			form $\boldsymbol{\Delta \varepsilon}^{ve^*}_{n+1}$ using 
			(\ref{eq:vec_epstar}).}
		\STATE{Compute the FFT $\mathcal{F}(\boldsymbol{\Delta 
				\varepsilon}^{ve^*}_{n+1})$.}
		\STATE{Compute $\mathbf{c}^{(\beta_E)}_{n+1}$ using (\ref{eq:vec_B}), 
			using the 
			known $b^{(\beta_E)}$ coefficients.}
		\STATE{Compute the FFT $\mathcal{F}(\mathbf{c}^{(\beta_E)}_{n+1})$.}
		\STATE{$Y^{ve}_{n+1} = -\frac{\mathbb{E}}{2 \Delta t^{\beta_E} 
				\Gamma(3-\beta_E)} \boldsymbol{\Delta 
				\varepsilon}^{ve^T}_{n+1}  
			\mathcal{F}^{-1}\left( \mathcal{F}(\mathbf{c}^{(\beta_E)}_{n+1}) 
			\odot 
			\mathcal{F} (\boldsymbol{\Delta \varepsilon}^{ve^*}_{n+1}) 
			\right)$.}
		\RETURN{$Y^{ve}_{n+1}$.}
	\end{algorithmic}
\end{algorithm}


\textbf{Computational Complexity of the Developed Scheme:} 
{Employing 
	(\ref{eq:fastY}) for the full time-fractional integration over 
	$\Omega$ yields a total computational complexity of $\mathcal{O}(N^2 \log 
	N)$, 
	similar to the $\mathcal{O}(N^2)$ complexity of the employed L1 FD scheme 
	for 
	fractional Caputo derivatives. Furthermore, the required storage for the 
	developed scheme is $\mathcal{O}(N)$.}

\section{Numerical Tests}
\label{Sec:NumericalResults}

{We present two qualitative examples with monotone and cyclic loads 
	for the SB free-energy density and the developed damaged, 
	visco-elasto-plastic model, where we verify the convergence and 
	computational complexity of the developed algorithms. For convergence 
	analyses, let $\mathbf{u}^*$ and $\mathbf{u}^\delta$ be, respectively, the 
	reference and approximate solutions in $\Omega =(0,T]$, for a specific 
	time-step size $\Delta t$. The global relative 
	error and convergence order are given, 
	respectively, as:}
\begin{equation}\label{eq:errordef}
\mathrm{err}(\Delta t) = \frac{\vert \vert \mathbf{u}^* - \mathbf{u}^\delta 
	\vert \vert_{L^\infty(\Omega)}}{\vert \vert \mathbf{u}^*  \vert 
	\vert_{L^\infty(\Omega)}}, \quad \mathrm{Order} = \log_2 
\left[\frac{\textrm{err}(\Delta t)}{\mathrm{err}(\Delta t / 2)}\right]
\end{equation}

We consider homogeneous initial conditions for all model variables in all 
cases. The presented algorithms were implemented in MATLAB R2019a and were run 
in a system with Intel Core i7-6700 CPU with 3.40 GHz, 16 GB RAM and Ubuntu 
18.04.2 LTS operating system.

\begin{example}[Convergence for Free-Energy Density]
	{We start with two convergence tests for the fractional 
		Helmholtz free-energy density using fabricated solutions. The first one 
		employs second-order increasing monotone strains, and the second uses 
		cyclic varying strains.}	
	
	\textbullet \,\textbf{Monotone Strains.} Let $t \in (0, T]$, with total 
	time $T 
	= 1\, [s]$. We define the quadratic strain {form $\varepsilon(t) = 
	\left(t/T\right)^2$. Therefore, analytical solution for the Helmholtz 
	free-energy (\ref{eq:ve_potential}) can be obtained directly as:}
	\begin{equation*}
		\psi^*(\varepsilon) = \frac{2^{2-\beta} \left[8+2^\beta \left(\beta - 
			5\right)\right]}{\Gamma(5-\beta)}\, T^4\, \mathbb{E}\, 
		\varepsilon^{2-\frac{\beta}{2}}.
	\end{equation*}
	We set $\mathbb{E} = 100\, [Pa.s^\beta]$, and 
	estimate the computational complexity of the direct (\ref{eq:slow}) and fast
	(\ref{eq:fast}) forms, with {varying} $\Delta t$. Figure \ref{fig:f1} 
	presents the {approximate free-energy} solution, 
	where we {recover the standard limit cases of a Hookean spring 
		($\beta \to 0$) and a Newtonian dashpot ($\beta \to 1$), as well as 
		second-order accuracy for the developed discretization.}
	\begin{figure}[t!]
		\centering
		\begin{subfigure}[b]{0.48\textwidth}
			\includegraphics[width=\columnwidth]{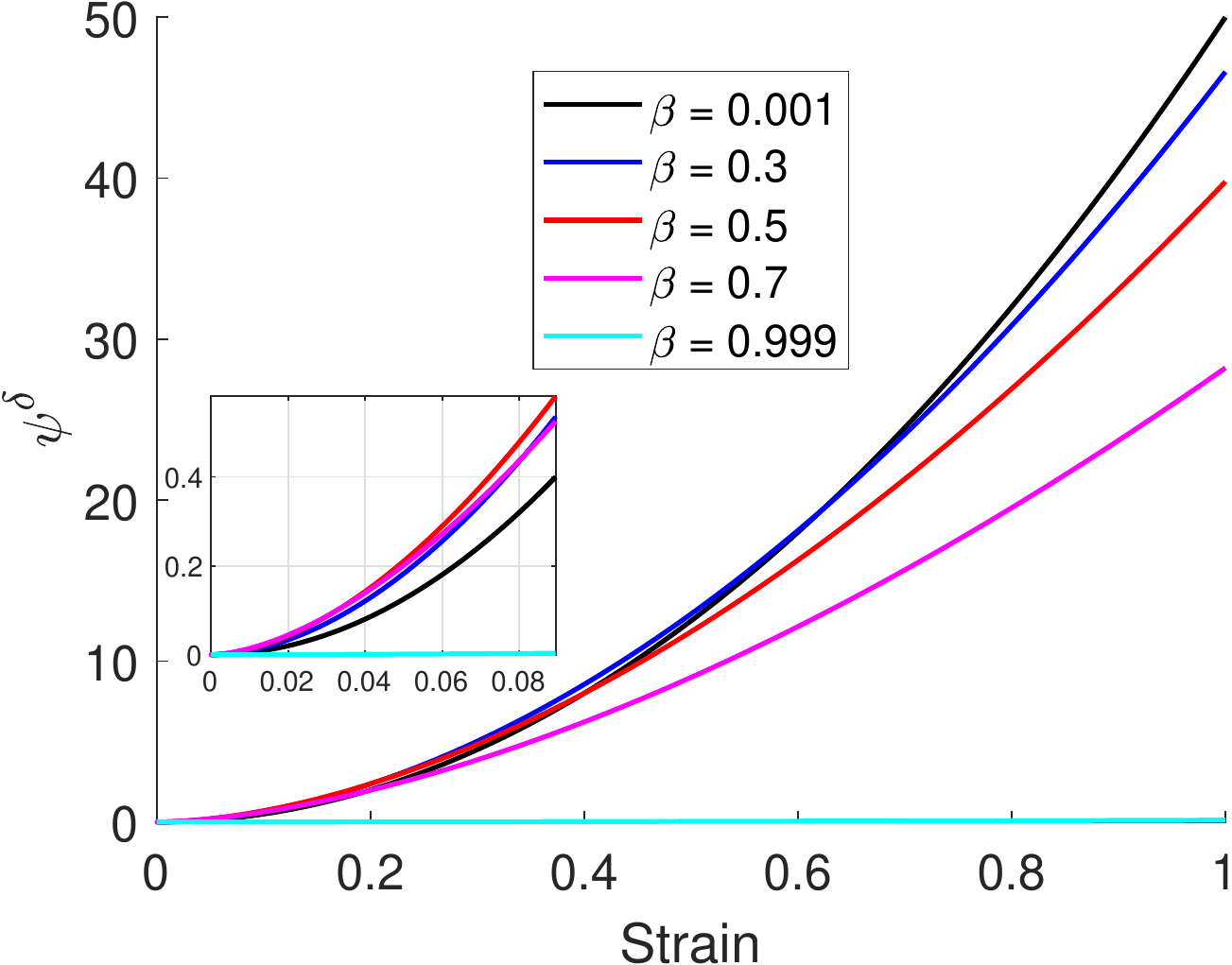}
		\end{subfigure}%
		\begin{subfigure}[b]{0.49\textwidth}
			\includegraphics[width=\columnwidth]{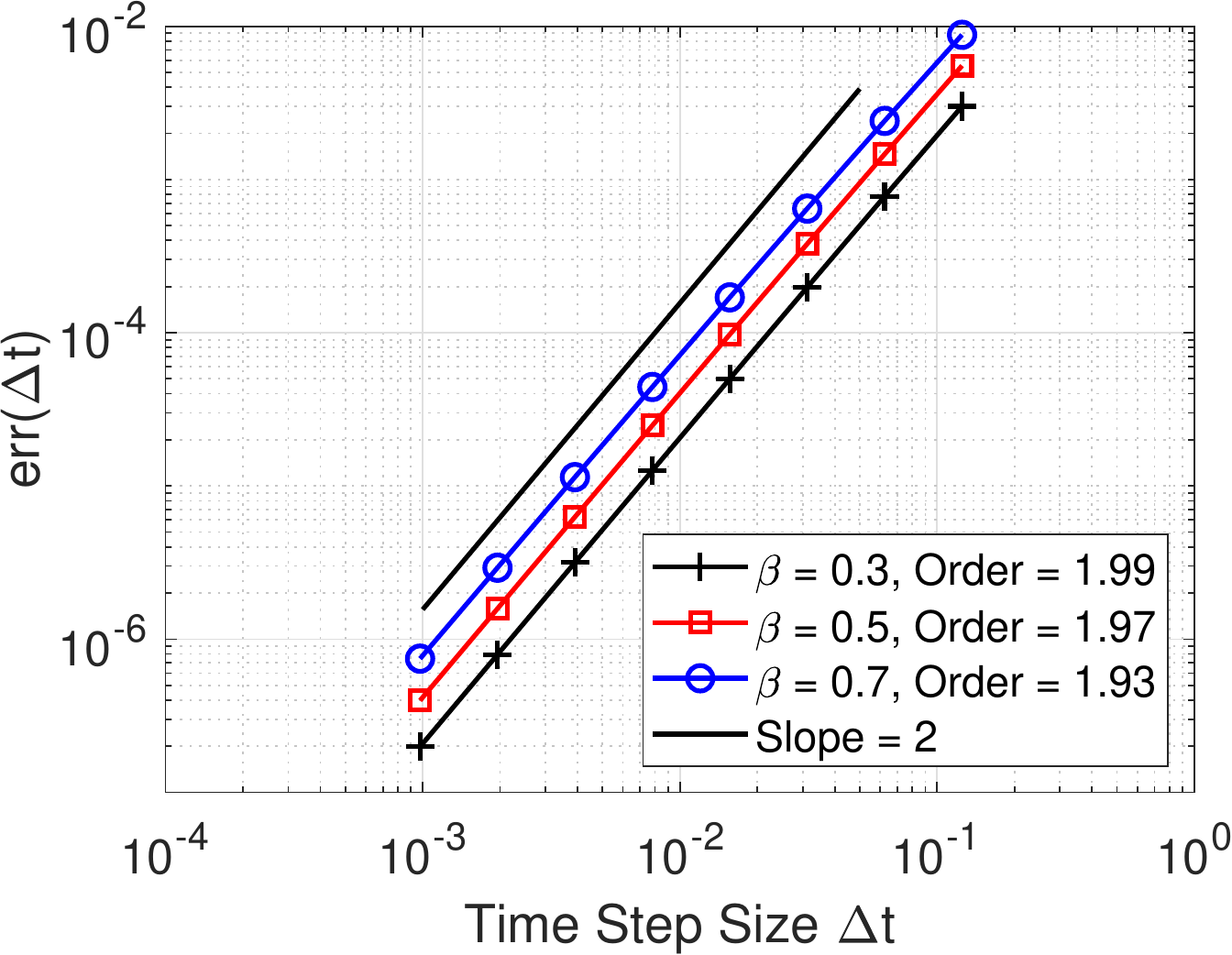}
		\end{subfigure}%
		\caption{Numerical results for the free-energy computation with a 
		quadratic 
			form for $\varepsilon(t)$. \textit{(Left)} $\psi^\delta$ 
			\textit{vs} strain 
			with varying $\beta$. \textit{(Right)} Relative error \textit{vs} 
			time-step 
			size for varying $\beta$, with second-order accuracy. 
			\label{fig:f1}}
	\end{figure}
	{Figure \ref{fig:f3} presents the obtained $\mathcal{O}(N^3)$ 
		and $\mathcal{O}(N^2 \log N)$ computational complexities, respectively, 
		for the 
		direct and FFT-based free-energy time-integration schemes. The 
		break-even point lies at $N=200$ time-steps.}
	\begin{figure}[t!]
		\centering
		\includegraphics[width=0.5\columnwidth]{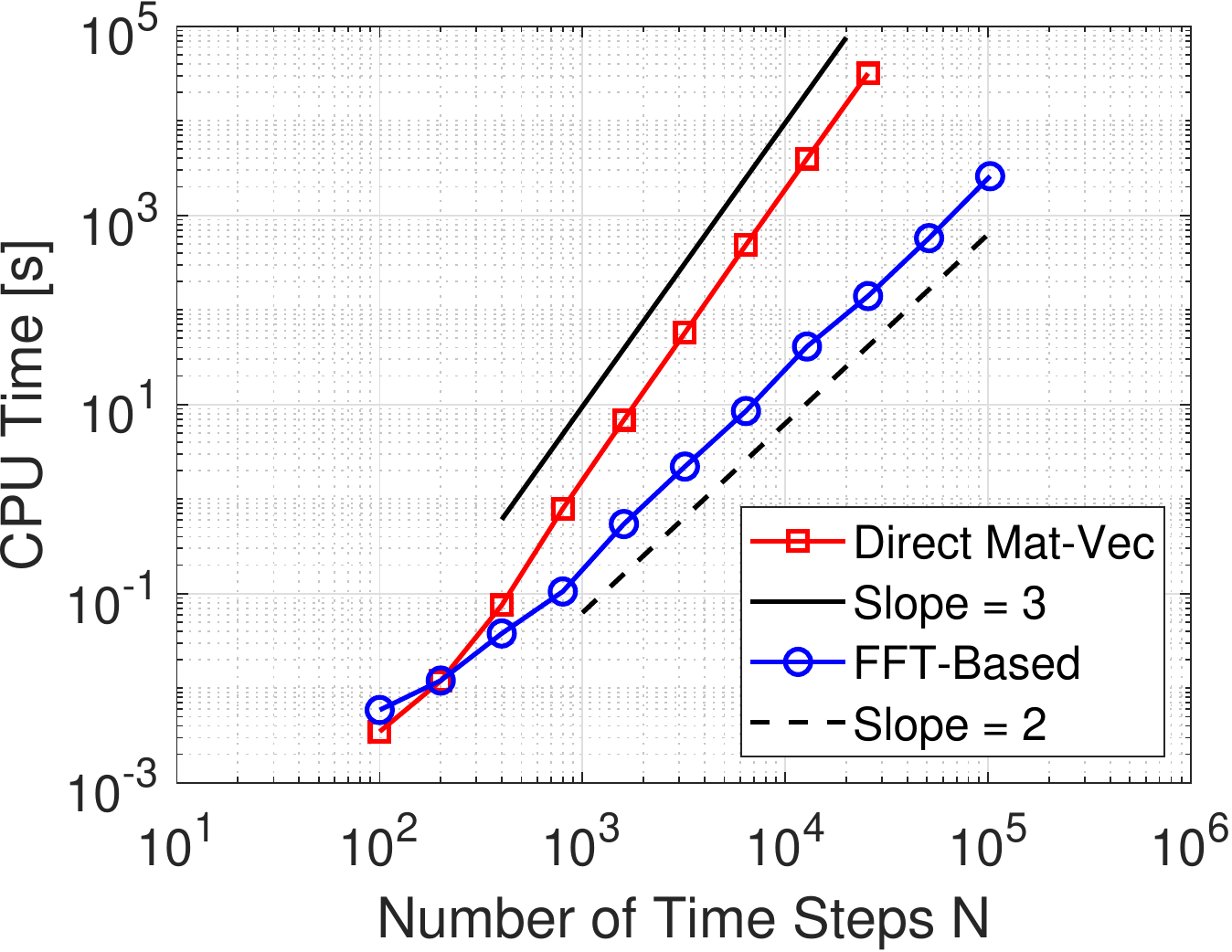}
		\caption{CPU time \textit{vs} number of time-steps of the 
			developed time-integration schemes for the fractional Helmholtz 
			free-energy density under monotone strains.
			\label{fig:f3}}
	\end{figure}
	
	\textbullet \,\textbf{Cyclic Strains.} We {utilize} a fabricated 
	sinusoidal strain solution $\varepsilon(t) = \varepsilon_0 \sin (\omega 
	t)$, 
	with $t \in (0,\,T]$, with amplitude $\epsilon_0$ and frequency $\omega$. 
	The 
	corresponding analytical solution for $\psi^*$ is cumbersome, and therefore 
	not 
	shown here. We set $\varepsilon_0 = 1$, $\omega = \pi\, [s^{-1}]$, $T = 
	50\, 
	[s]$, $\beta = 0.5$ and $\mathbb{E} = 1\, [Pa.s^{0.5}]$, {and 
		start with a sufficient number of time-steps to capture the oscillation 
		modes.} 
	Figure 
	\ref{fig:f4} illustrates the obtained results, where we capture the highly 
	oscillatory behavior for both transient and steady-state parts with 
	second-order accuracy.
	\begin{figure}[t!]
		\centering
		\begin{subfigure}[b]{0.5\textwidth}
			\includegraphics[width=\columnwidth]{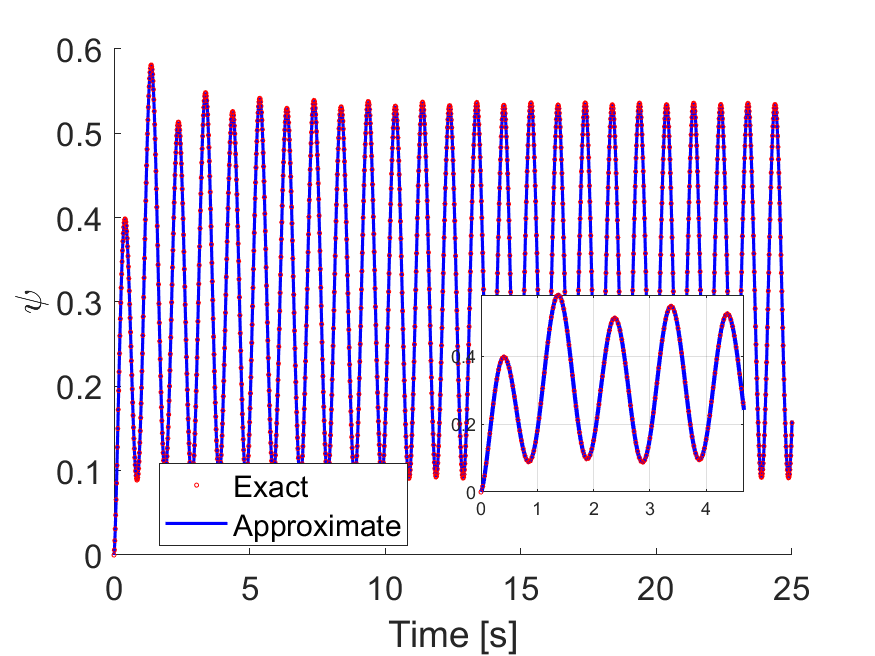}
			\caption{Free-energy \textit{vs} time.}		
		\end{subfigure}%
		\begin{subfigure}[b]{0.46\textwidth}
			\includegraphics[width=\columnwidth]{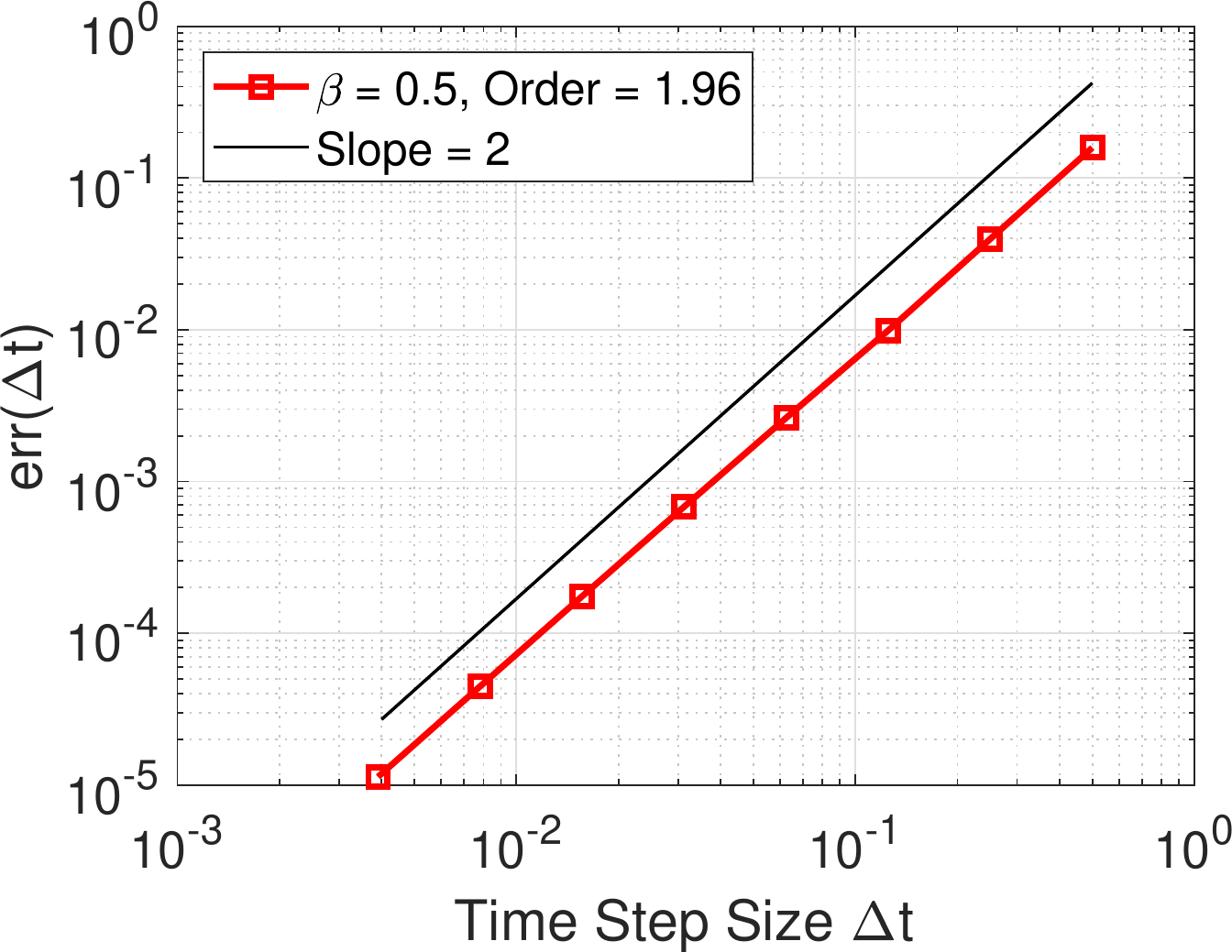}
			\caption{Relative error \textit{vs} time-step size.}		
		\end{subfigure}%
		\caption{\textit{(A)} Free-energy density computations for 
			cyclic strains \textit{vs} time, $N=3200$ time-steps and $\beta = 
			0.5$. \textit{(B)} Convergence analysis showing second-order 
			accuracy. 
			\label{fig:f4}}
	\end{figure}
\end{example}

\begin{example}[{Fractional Visco-Elasto-Plastic Model with 
		Damage}]
	
	We test our developed model and fractional return-mapping 
	algorithm subject to prescribed {monotone/cyclic strains}. 
	{The convergence analysis is done with a benchmark solution and 
		we analyze the quality of the anomalous damage response with respect to 
		the 
		fractional orders $\beta_E,\beta_K$ from visco-elasticity/plasticity 
		under
		different strain amplitudes/frequencies.}
	
	\textbullet\, \textbf{Monotone Strains.} Let $\varepsilon(t) = 
	\dot{\varepsilon} t$, where $t \in (0,\,T]$, final time $T=0.03125\, [s]$ 
	and 
	strain rate $\dot{\varepsilon} = 0.64\,[s^{-1}]$, {and therefore 
		$\varepsilon(T) = 0.02$. We set} $\beta_E 
	= 0.5$, $\mathbb{E} = 50\,[Pa.s^{0.5}]$, $\mathbb{K} = 
	10\,[Pa.s^{\beta_K}]$, 
	$\tau^Y = 1\,[Pa]$, $S = 10^{-4}\,[Pa]$ and $s=1$. A benchmark solution for 
	the 
	stress (\textit{see Fig.\ref{fig:stress_strain_monotone}}) is computed with 
	time-step size $\Delta t = 2^{-20}\,[s]$ and varying fractional orders 
	$\beta_K$, {where we observe that higher values for $\beta_K$ 
		led to increased hardening and damage for the prescribed strain rate. }
	\begin{figure}[t!]
		\centering
		\includegraphics[width=0.49\columnwidth]{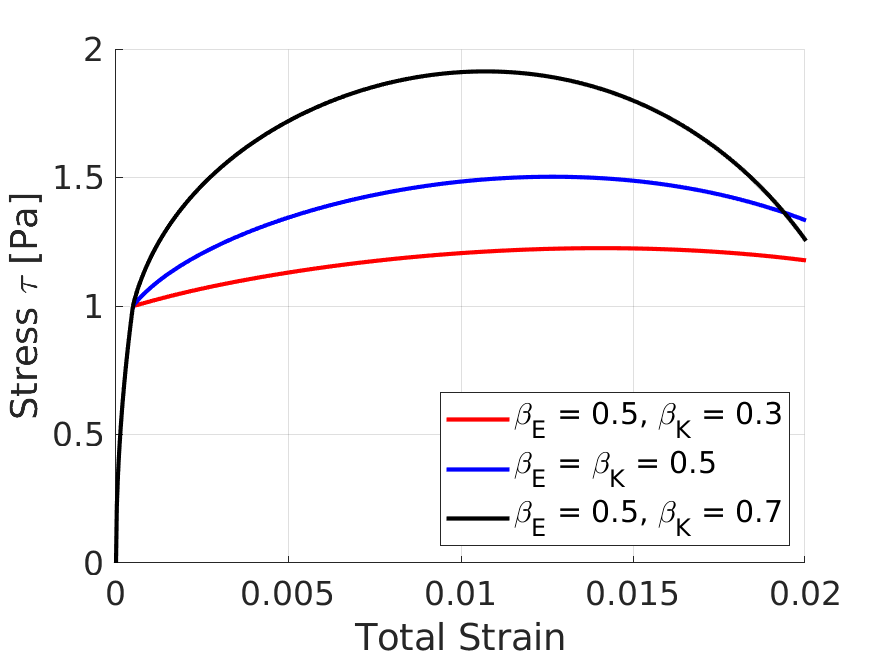}
		\caption{Stress \textit{vs} strain for the benchmark solution with 
			time-step size $\Delta t = 2^{-20}$, $\beta_E = 0.5$ and different 
			$\beta_K$ values. \label{fig:stress_strain_monotone}}
	\end{figure}
	{We observe a linear convergence rate in Figure \ref{fig:conv_M1}, due to 
	the 
		employed backward-Euler discretization in the fractional return-mapping 
		algorithm. A second-order computational complexity for the fractional 
		return-mapping algorithm is also verified in Figure 
		\ref{fig:time_M1}.}
	\begin{figure}[t!]
		\centering
		\begin{subfigure}[b]{0.49\textwidth}
			\includegraphics[width=\columnwidth]{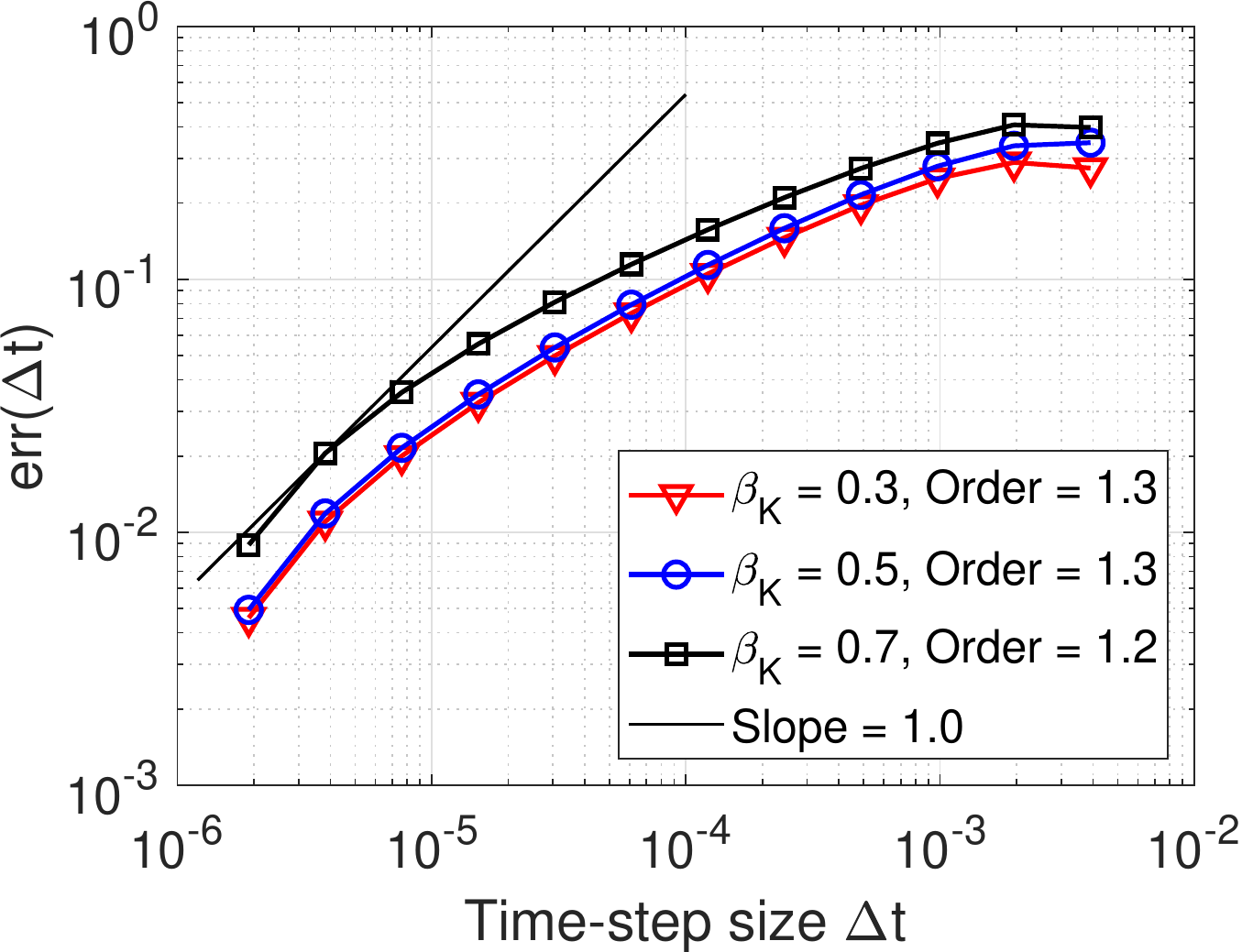}
			\caption{Convergence behavior.\label{fig:conv_M1}}		
		\end{subfigure}%
		\begin{subfigure}[b]{0.485\textwidth}
			\includegraphics[width=\columnwidth]{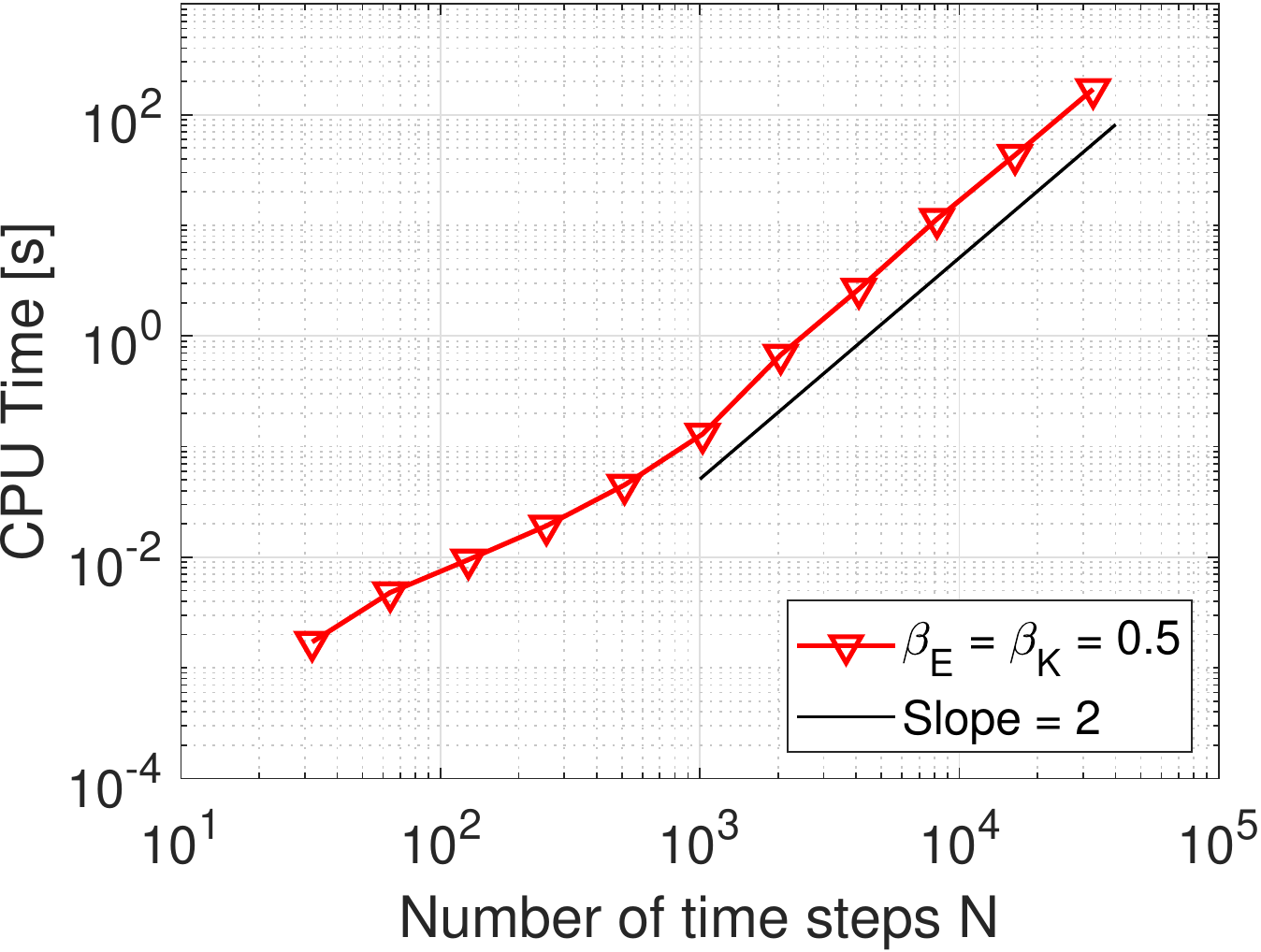}
			\caption{Computational complexity.\label{fig:time_M1}}		
		\end{subfigure}%
		\caption{Fractional visco-elasto-plastic model with damage under 
		monotone 
			strains. \textit{(A)} First-order convergence behavior. 
			\textit{(B)} 
			Computational time \textit{vs} number of time-steps, with 
			second-order 
			computational complexity. \label{fig:conv_time_M1}}
	\end{figure}
	The influence of {hardening} and visco-elastic damage energy release rate 
	is shown in Figure \ref{fig:Damage_Y_monotone}. {We observe that higher 
	damage values are obtained for $\beta_K = 0.7$, despite the higher 
	accumulated plastic strains for lower values of $\beta_K$. The higher 
	damage is instead due to higher values of damage energy release rates shown 
	in Figure \ref{fig:Y_ve_monotone} for $\beta_K = 0.7$. We note that similar 
	to the stress-strain response, the visco-elastic fractional free-energy is 
	power-law 
		memory-dependent on the strain rates, therefore leading to the observed 
		anomalous behavior.}
	\begin{figure}[t!]
		\centering
		\begin{subfigure}[b]{0.49\textwidth}
			\includegraphics[width=\columnwidth]{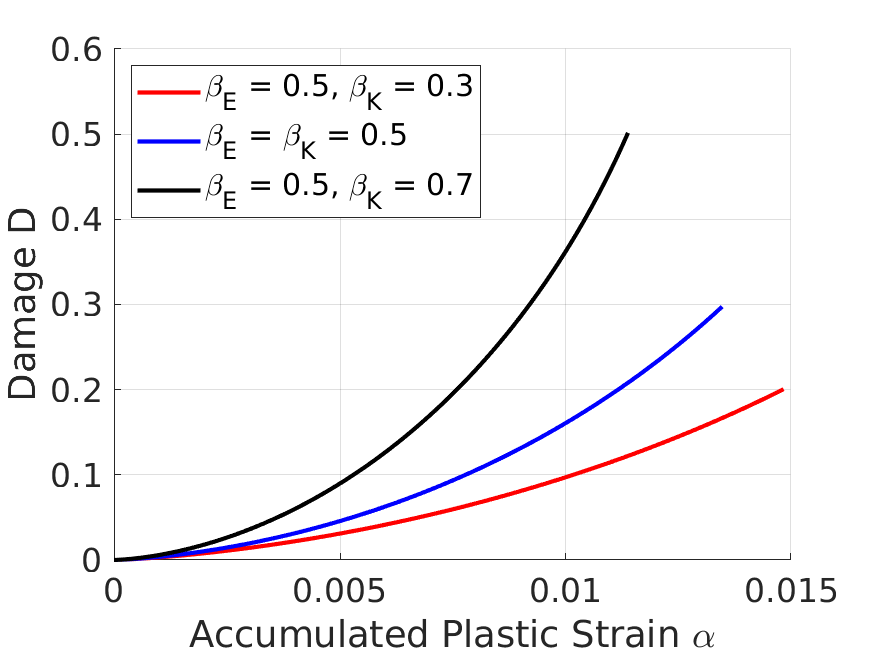}
			\caption{\label{fig:Damage_alpha_monotone}}		
		\end{subfigure}%
		\begin{subfigure}[b]{0.49\textwidth}
			\includegraphics[width=\columnwidth]{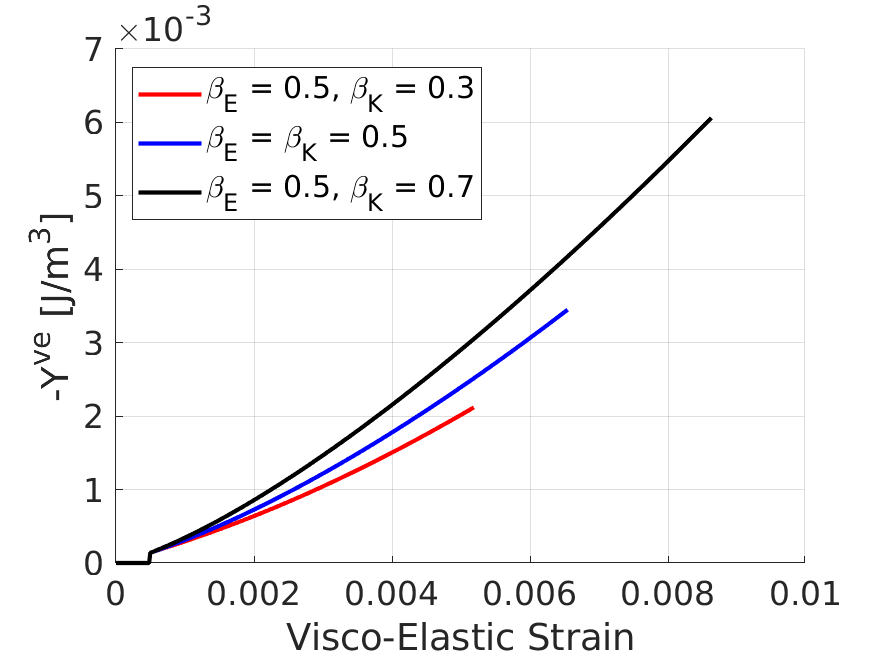}
			\caption{\label{fig:Y_ve_monotone}}		
		\end{subfigure}%
		\caption{Developed model under monotone strains: \textit{(A)} Damage 
			\textit{vs} accumulated plastic strain, with higher damage but less 
			plasticity for higher $\beta_K$. \textit{(B)} Damage energy release 
			rate 
			\textit{vs} visco-elastic strains, which are both larger for higher 
			values 
			of fractional order $\beta_K$. \label{fig:Damage_Y_monotone}}
	\end{figure}
	
	\textbullet\, \textbf{Cyclic Strains.} {To investigate the interplay 
	between 
		the damage/hardening/\\viscosity and hysteresis effects, we perform a 
		constant 
		rate loading/unloading cyclic strain test, mathematically expressed as:}
	\begin{equation*}
		\varepsilon(t) = \frac{2 \varepsilon_A}{\pi} 
		\arcsin\left(\sin\left(2\pi\omega 
		t\right)\right),
	\end{equation*}
	where $\varepsilon_A$ and $\omega$ represent, respectively, the amplitude 
	and 
	frequency of {total strains. Here, we focus on low-cycle 
		fatigue behavior, and therefore we set} $\varepsilon = 0.1$, and three 
		strain frequencies $\omega = \lbrace 
	2\pi,\,4\pi,\,8\pi\rbrace\,[s^{-1}]$, which correspond, respectively, to 
	approximate absolute strain rates of $|\dot{\varepsilon}| \approx \lbrace 
	2.51,\,5.02,\,10.05\rbrace$. {We set a total time $T = 10\,[s]$, 
		and for each frequency, we use} $N = \lbrace 
	8\,000,\,16\,000,\,32\,000\rbrace$ time-steps, corresponding to $\Delta t = 
	\lbrace 1.25\times 10^{-3},\, 6.25\times 10^{-4},\,3.125\times 
	10^{-4}\rbrace\,[s]$. The material parameters are set to 
	$\mathbb{E} = 25\,[Pa.s^{\beta_E}]$, $\mathbb{K} = 10\,[Pa.s^{\beta_K}]$, 
	$\tau^Y = 1\,[Pa]$, $S = 1\,[Pa]$ and $s = 1$, {where we set the fractional 
		order values} $\beta_E = \beta_K = \lbrace 0.3,\,0.5,\,0.7\rbrace$.
	
	{The stress-strain hysteresis results are presented in Figure 
		\ref{fig:Stress_cyclic}. We observe that higher frequencies led to 
		more softening in the model, while higher values of fractional orders 
		$\beta_E$, $\beta_K$ led to increased hardening, followed by 
		softening}. 
	\begin{figure}[t!]
		\centering
		\begin{subfigure}[b]{0.33\textwidth}
			\includegraphics[width=\columnwidth]{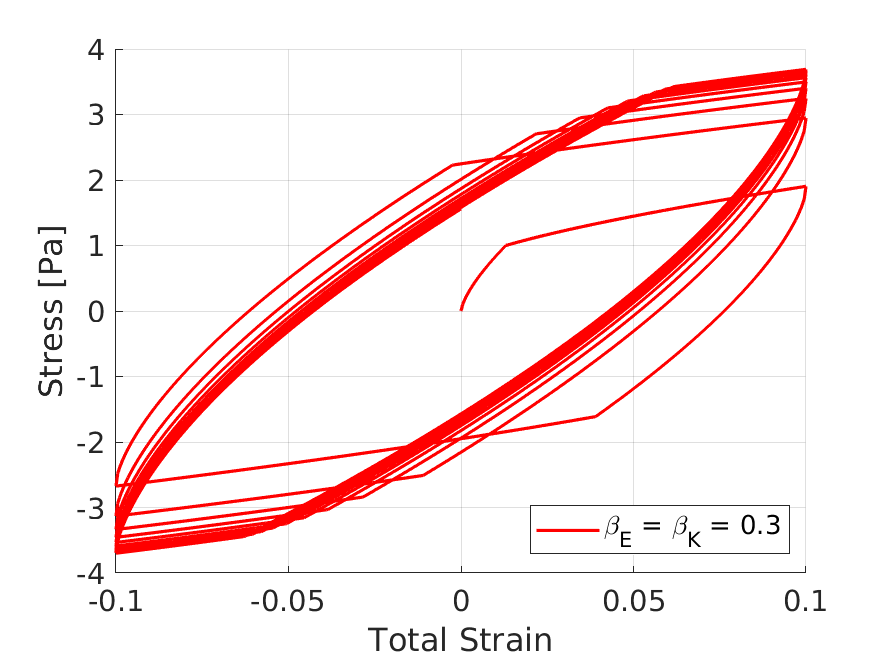}
			\caption{$\beta_E = \beta_K = 0.3$.}		
		\end{subfigure}%
		\begin{subfigure}[b]{0.33\textwidth}
			\includegraphics[width=\columnwidth]{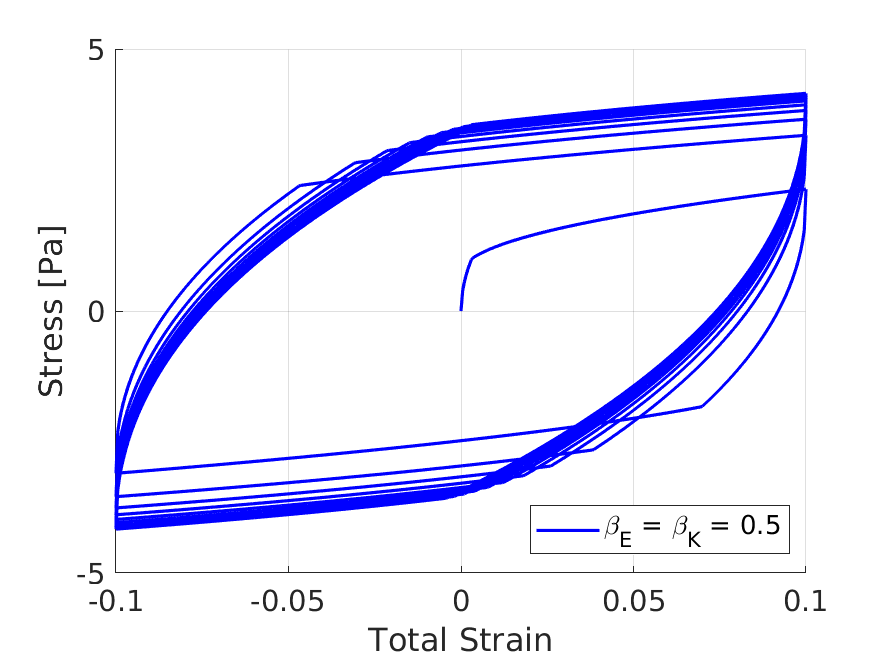}
			\caption{$\beta_E = \beta_K = 0.5$.}		
		\end{subfigure}%
		\begin{subfigure}[b]{0.33\textwidth}
			\includegraphics[width=\columnwidth]{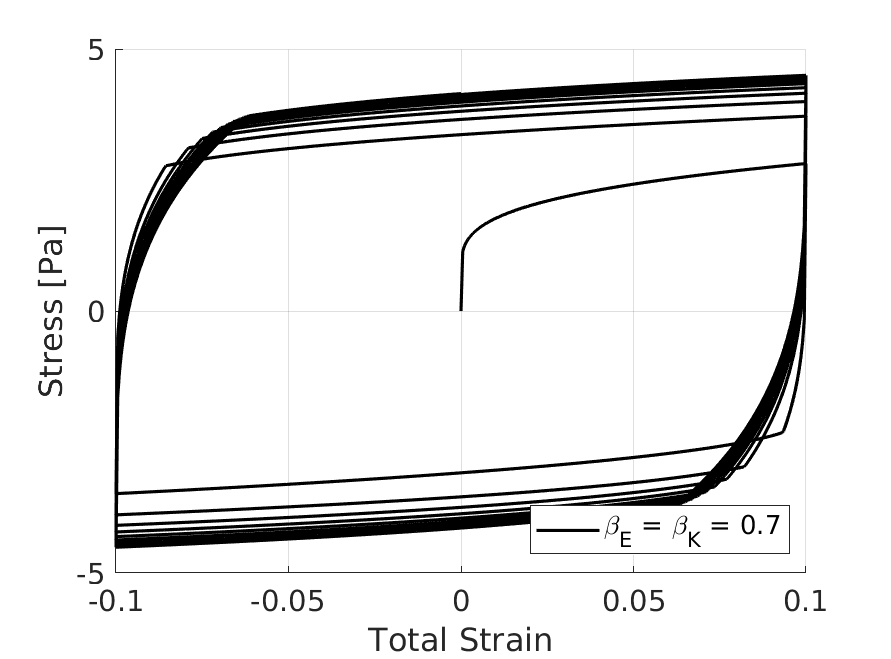}
			\caption{$\beta_E = \beta_K = 0.7$.}		
		\end{subfigure} \\
		\begin{subfigure}[b]{0.33\textwidth}
			\includegraphics[width=\columnwidth]{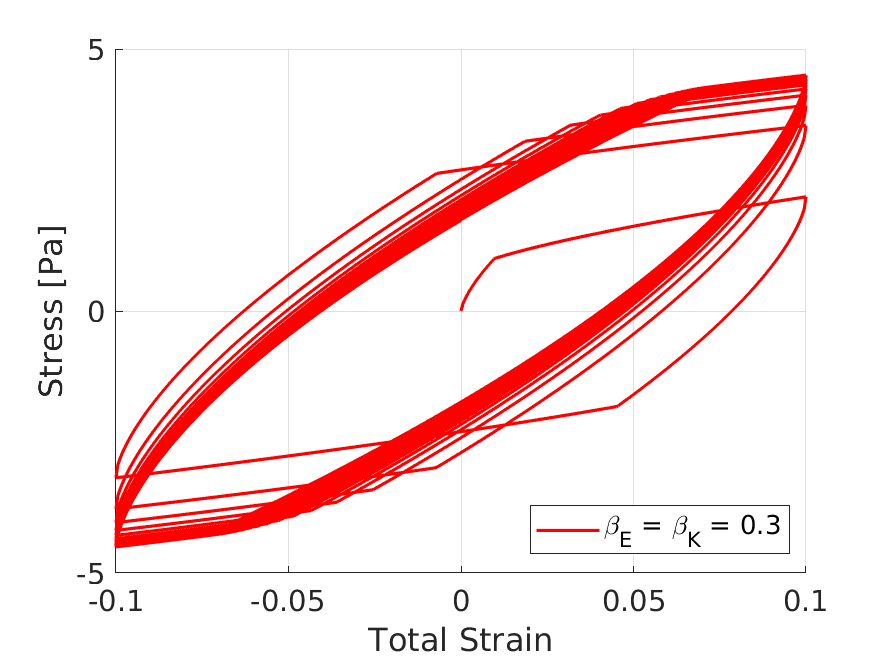}
			\caption{$\beta_E = \beta_K = 0.3$.}		
		\end{subfigure}%
		\begin{subfigure}[b]{0.33\textwidth}
			\includegraphics[width=\columnwidth]{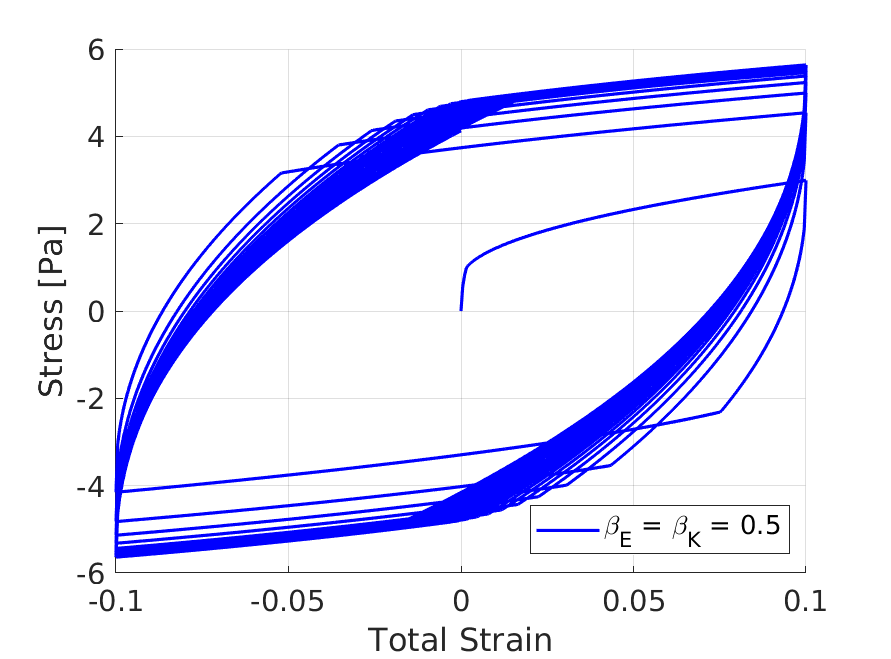}
			\caption{$\beta_E = \beta_K = 0.5$.}		
		\end{subfigure}%
		\begin{subfigure}[b]{0.33\textwidth}
			\includegraphics[width=\columnwidth]{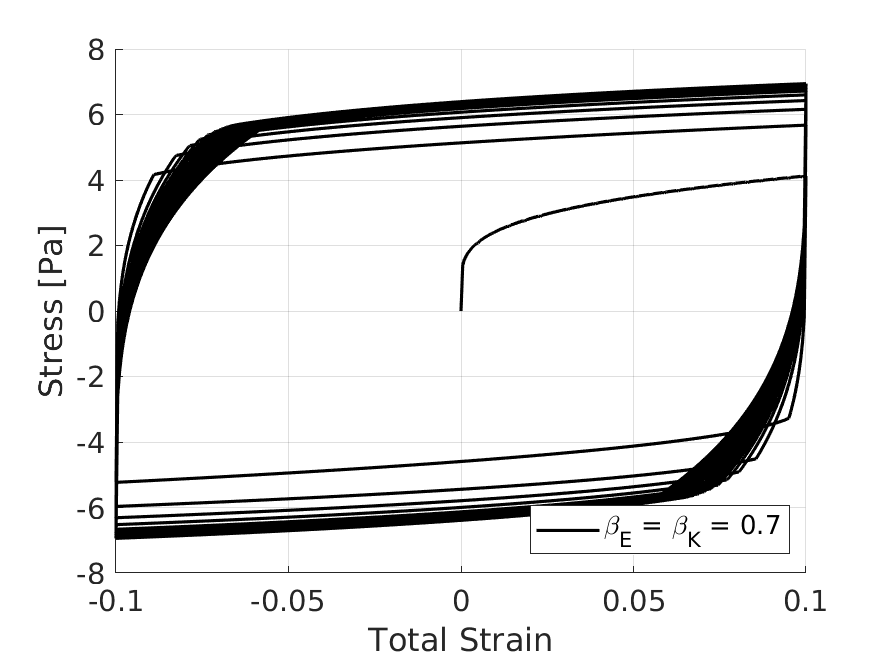}
			\caption{$\beta_E = \beta_K = 0.7$.}		
		\end{subfigure} \\
		\begin{subfigure}[b]{0.33\textwidth}
			\includegraphics[width=\columnwidth]{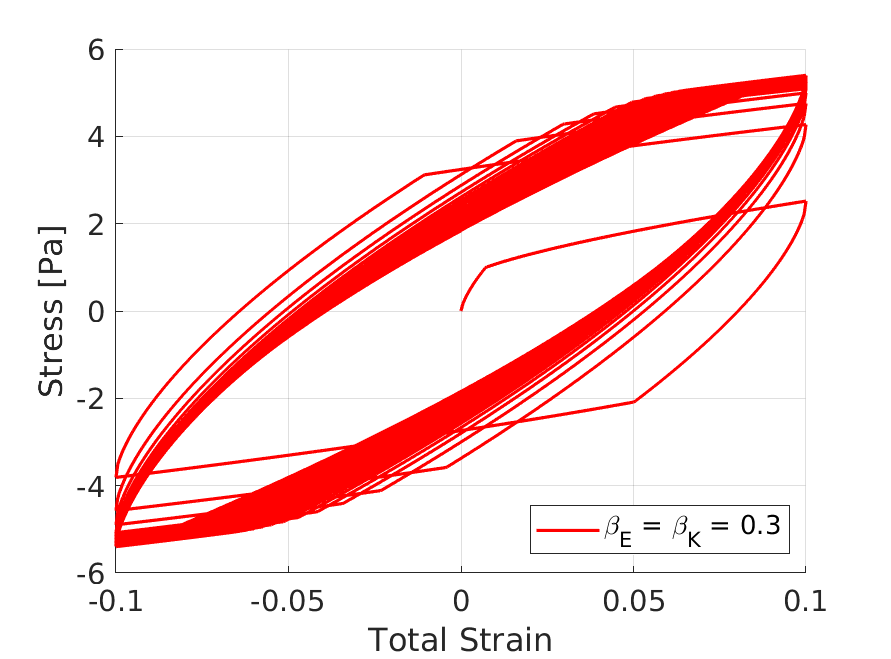}
			\caption{$\beta_E = \beta_K = 0.3$.}		
		\end{subfigure}%
		\begin{subfigure}[b]{0.33\textwidth}
			\includegraphics[width=\columnwidth]{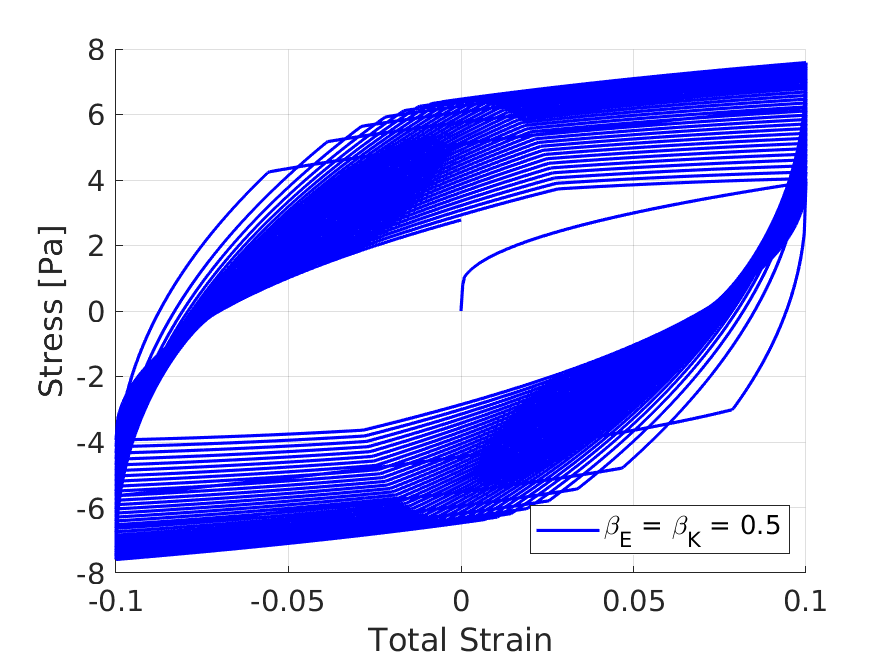}
			\caption{$\beta_E = \beta_K = 0.5$.}		
		\end{subfigure}%
		\begin{subfigure}[b]{0.33\textwidth}
			\includegraphics[width=\columnwidth]{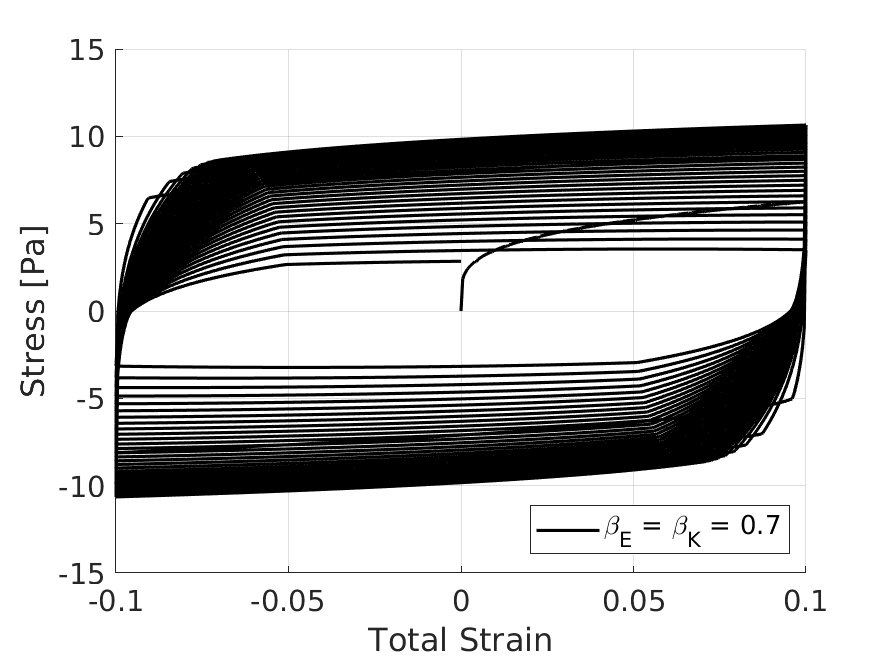}
			\caption{$\beta_E = \beta_K = 0.7$.}		
		\end{subfigure} 
		\caption{Stress hysteresis response for cyclic strains with frequencies 
			\textit{(A)-(C)} $\omega = 2\pi$, \textit{(D)-(F)} 
			$\omega = 4\pi$, \textit{(G)-(I)} $\omega = 8\pi$. 
			\label{fig:Stress_cyclic}}
	\end{figure}
	{Such damage increase is illustrated in Fig. \ref{fig:Damage_cyclic}, where 
	we 
		observe that higher 
		$\beta_E$ and $\beta_K$ values led to increased plasticity for all 
		cases, 
		with a significant} increase of damage rates for $\beta_E = \beta_K = 
		0.5, 0.7$ 
	when $\omega = 8\pi$. We also observe from Fig. \ref{fig:Y_cyclic} that due 
	to 
	the anomalous nature of the fractional visco-elastic free-energy potential, 
	the 
	damage energy release rates substantially increase with higher fractional 
	orders and loading rates, which contribute to the observed higher values of 
	damage. Therefore, for this model, {higher} material viscosity in both 
	visco-elastic and visco-plastic parts might be sufficient to yield lower 
	values 
	of damage {at low frequencies due to internal dissipation mechanisms, but 
	at 
		higher frequencies and therefore more loading 
		cycles, they lead} to earlier material failure.
	\begin{figure}[t!]
		\centering
		\begin{subfigure}[b]{0.33\textwidth}
			\includegraphics[width=\columnwidth]{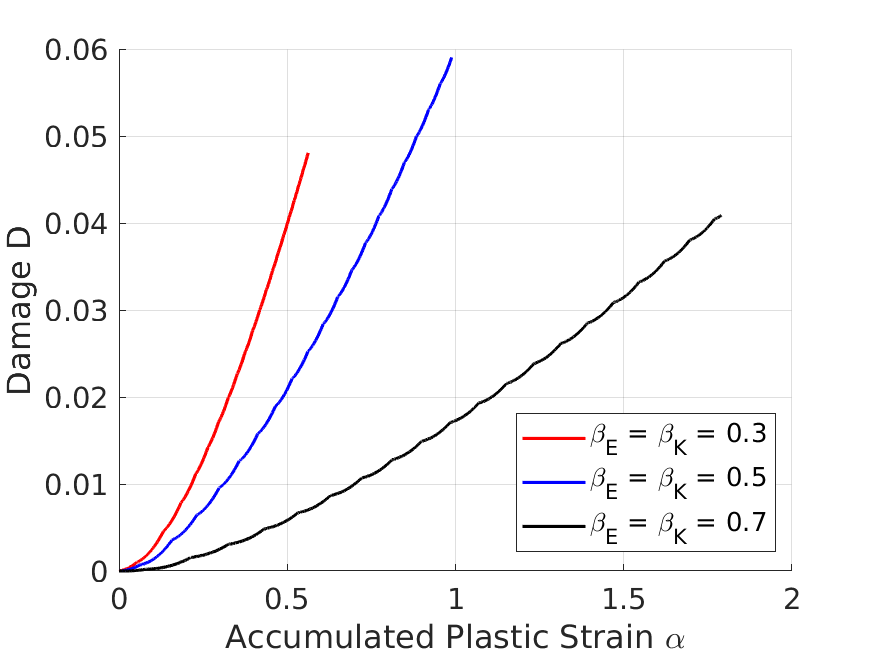}
			\caption{$\omega = 2\pi$.}		
		\end{subfigure}%
		\begin{subfigure}[b]{0.33\textwidth}
			\includegraphics[width=\columnwidth]{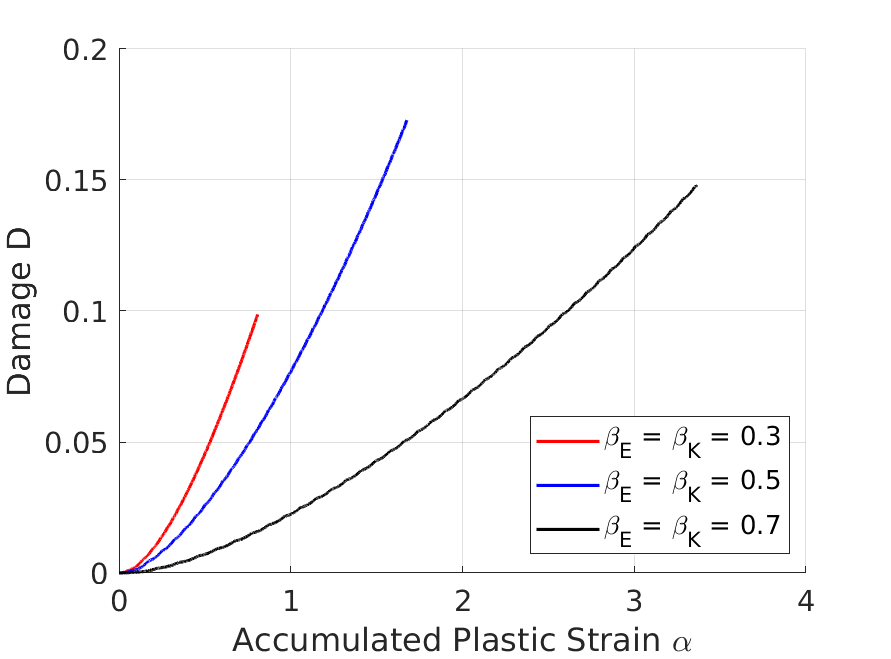}
			\caption{$\omega = 4\pi$.}		
		\end{subfigure}%
		\begin{subfigure}[b]{0.33\textwidth}
			\includegraphics[width=\columnwidth]{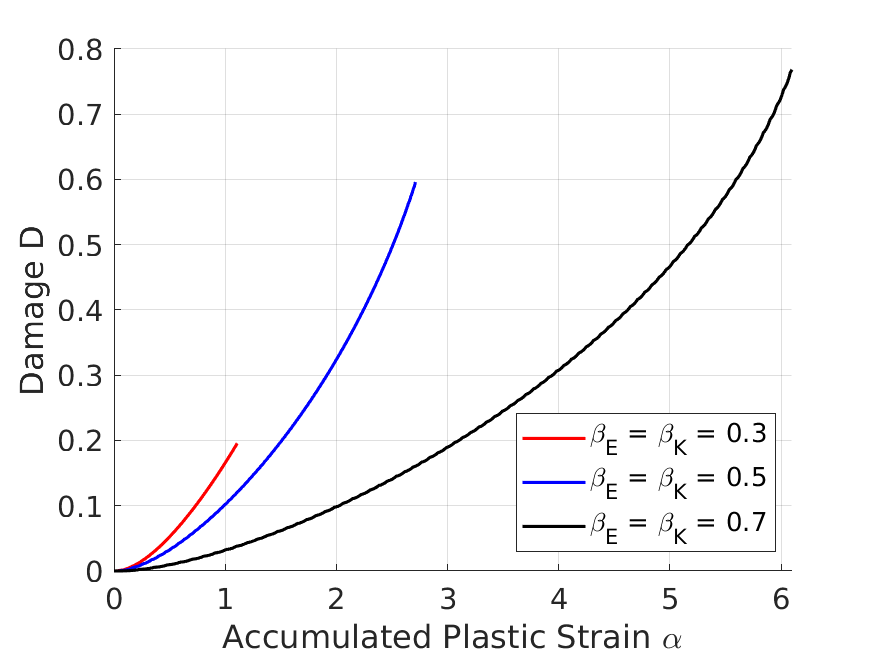}
			\caption{$\omega = 8\pi$.}		
		\end{subfigure}
		\caption{Damage \textit{vs} accumulated plastic strains with varying 
		strain 
			frequencies. \label{fig:Damage_cyclic}}
	\end{figure}
	\begin{figure}[t!]
		\centering
		\begin{subfigure}[b]{0.333\textwidth}
			\includegraphics[width=\columnwidth]{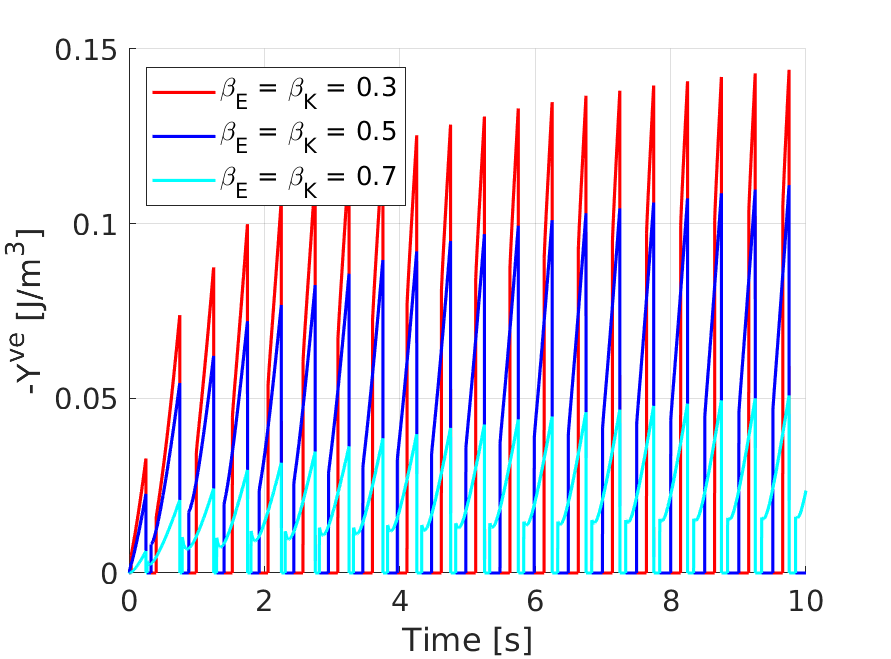}
			\caption{$\omega = 2\pi$.}		
		\end{subfigure}%
		\begin{subfigure}[b]{0.333\textwidth}
			\includegraphics[width=\columnwidth]{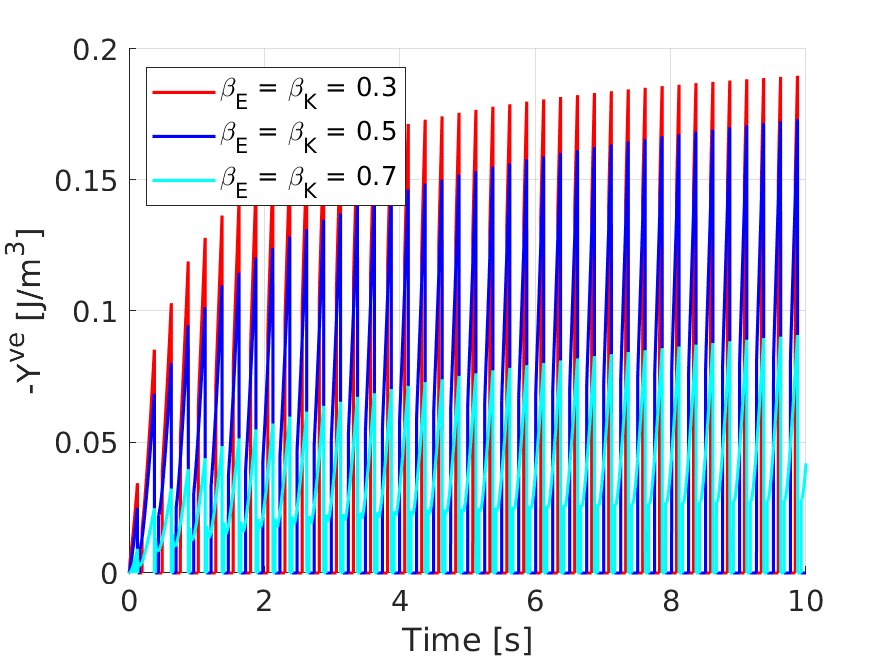}
			\caption{$\omega = 4\pi$.}		
		\end{subfigure}%
		\begin{subfigure}[b]{0.333\textwidth}
			\includegraphics[width=\columnwidth]{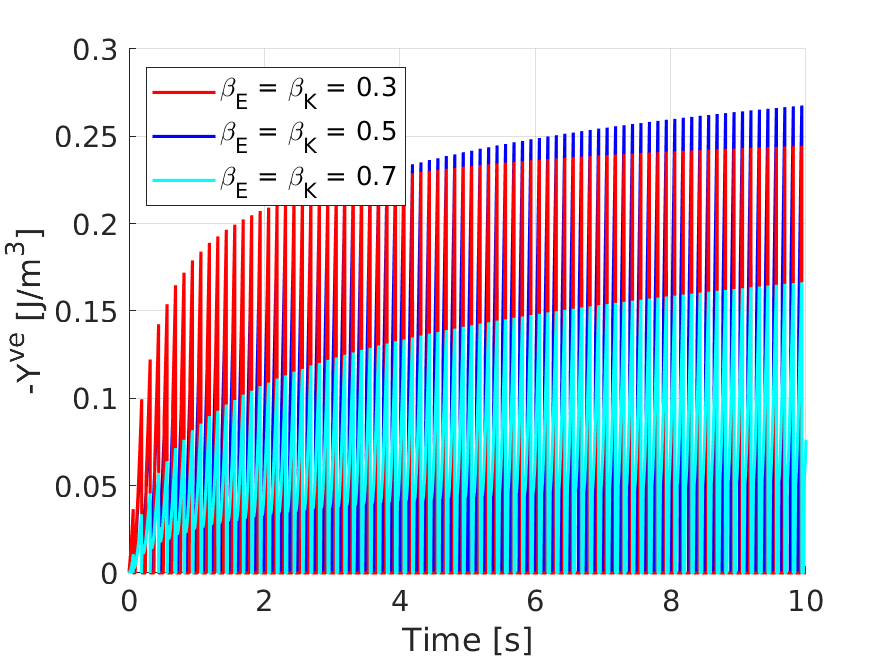}
			\caption{$\omega = 8\pi$.}		
		\end{subfigure}
		\caption{Damage energy release rate \textit{vs} time for the total 
		strain 
			with varying frequency. \label{fig:Y_cyclic}}
	\end{figure}
	
\end{example}

\section{Conclusions}
\label{Sec:Conclusions}

{We developed a thermodynamically consistent, fractional 
	vis\-co-elasto-plastic model with memory-dependent damage using fractional 
	Helmholtz free-energies, visco-plastic/damage 
	potentials and the Clausius-Duhem inequality}. The damage energy release 
	rate 
was {derived from the visco-elastic free-energy to obtain a 
	consistent bulk energy loss for anomalous materials.}

{A first-order, semi-implicit fractional return-mapping 
	algorithm, which generalizes existing standard ones, was developed to solve 
	the 
	resulting nonlinear system of FDEs. We note that most 
	existing algorithms for standard plasticity models are not more accurate 
	than 
	ours. We also developed a new FD scheme with accuracy 
	$\mathcal{O}(\Delta 
	t^{2-\beta})$ for the free-energy/damage energy release, with computational 
	complexity of $\mathcal{O}(N^2 \log N)$ through FFTs.}

We {also} performed a set of numerical tests and observed that:
\begin{itemize}
	\item The {fractional} orders $\beta_E$ and $\beta_K$ 
	{tune} the competition between the plastic slip and 
	damage energy 
	release rate for {damage} evolution.
	
	\item Higher values of $\beta_E$, $\beta_K$ yielded lower damage 
	levels for lower strain rates and cycles; However, the damage increased 
	{significantly faster than lower values of $\beta_E, 
		\beta_k$ for} higher strain 
	rates/cycles.
	
	\item For the free-energy discretization, the break-even point between the 
	original and fast schemes was low, about $N=200$ time-steps.
	
	\item {The developed discretization recovered the limit 
		Hookean $\beta \to 
		0$ and Newtonian $\beta \to 1$ cases for the free-energy.}
\end{itemize}

{In the presence of} single- to multi- singularities, 
{the accuracy of the developed scheme can improve} through 
a variant of a self-singularity-capturing approach \cite{Suzuki2018Singularity} 
for all fractional operators. Nevertheless, non-smooth loading/unloading 
conditions pose additional challenges to develop high-order schemes for the 
model. {In} terms of efficiency, the 
computational bottleneck lies in the free-energy discretization, which 
{needs further improvements before employing fast schemes} for 
the fractional derivatives, \textit{e.g.}, fast convolution 
\cite{Zheng2017, Lubich2002} and fast multi-pole approaches \cite{Vikram2010}. 
{Variants} of the developed model can be incorporated in a 
straightforward fashion. The visco-elastic part could be composed of any 
{data/design-driven} arrangement of SB elements, 
{\textit{e.g.} Kelvin-Voigt, Maxwell, Kelvin-Zener 
	\cite{Schiessel1995}, while adding the corresponding energy release rates 
	to 
	the damage potential. In addition, similar frameworks involving fractional 
	damage energy release rates can be employed to phase field models 
	\cite{boldrini2016non}. Potential applications of the developed work could 
	be, 
	\textit{e.g.}, failure of polymers, bio-tissues, and ductile 
	metals, where the fractional-orders $\beta_E$, $\beta_K$ can be related to 
	the 
	evolving fractal-like microstructure \cite{Mashayekhi2019Fractal}. The 
	presented model could also be employed in the context of nonlinear 
	dynamics} 
of mechanical systems \cite{Suzuki2014,varghaei2019Vibration}.

Finally, the employment of nonlocal truncated time derivatives \cite{Du2017} 
and potentials could have additional impacts on reducing the computational 
complexity of the developed scheme, due to the shorter memory. Furthermore, the 
use of such operators seems particularly interesting to 
naturally address the ``memory reset" for internal variables such as the 
hardening $\alpha$ for hysteresis loading \cite{Suzuki2016}.

\appendix
\section{Proof of Lemma \ref{lemma:stress_strain}}
\label{Ap:lemma_stress_strain}

We take the time derivative of the free-energy 
(\ref{eq:ve_potential}) and obtain:
\begin{equation}\label{eq:proofSB_1}
\dot{\psi} = \int^\infty_0 \tilde{E}(z) \int^t_0 \exp 
\left(-\frac{t-s}{z}\right) \dot{\varepsilon}(s) ds \left( \frac{d}{dt} 
\int^t_0 \exp \left(-\frac{t-s}{z}\right) \dot{\varepsilon}(s) ds \right) dz
\end{equation}
with
\begin{equation}\label{eq:proofSB_2}
\frac{d}{dt} \int^t_0 \exp \left(-\frac{t-s}{z}\right) \dot{\varepsilon}(s) ds 
= \dot{\varepsilon}(t) - \int^t_0\frac{1}{z}\exp\left(-\frac{t-s}{z}\right) 
\dot{\varepsilon}(s) ds.
\end{equation}
Substituting (\ref{eq:proofSB_2}) into (\ref{eq:proofSB_1}), we obtain:
\begin{align}
\dot{\psi} = & \left[ \int^\infty_0 \tilde{E}(z) \left( \int^t_0 \exp 
\left(-\frac{t-s}{z}\right) \dot{\varepsilon}(s) ds\right) dz\right] 
\dot{\varepsilon} \nonumber \\
\qquad \qquad {} & \qquad - \int^\infty_0 \frac{\tilde{E}(z)}{z} \left(\int^t_0 
\exp\left(-\frac{t-s}{z}\right)\dot{\varepsilon}(s) ds\right)^2 dz.
\label{eq:proofSB_3}
\end{align}
Let $\mathbb{E}^* = \frac{\mathbb{E}}{\Gamma(1-\beta)\Gamma(\beta)}$. 
{Note that the term inside brackets in (\ref{eq:proofSB_3}) 
	equals:}
{\small
	\begin{align}
	{} & \int^\infty_0 \tilde{E}(z) \left( \int^t_0 
	\exp\left(-\frac{t-s}{z}\right) 
	\dot{\varepsilon}(s)\,ds \right)\, dz \nonumber \\
	& = \int^\infty_0 \frac{\mathbb {E}^*}{z^{\beta + 1}} \left( \int^t_0 
	\exp\left(-\frac{t-s}{z}\right) \dot{\varepsilon}(s)\,ds \right)\, dz 
	\nonumber \\
	& = \mathbb{E}^* \int^t_0 \left[ \int^\infty_0 z^{-(\beta + 1)} 
	\exp\left(-\frac{t-s}{z}\right)\, dz \right]\, \dot{\varepsilon}(s)\,ds 
	\nonumber \\
	& = \mathbb{E}^* \int^t_0 \left[ \int^\infty_0 \frac{u^{\nu - 
			1}}{(t-s)^{\beta}} \exp(-u) \, du \right]\, 
			\dot{\varepsilon}(s)\,ds 
	\nonumber \\
	& = \mathbb{E}^* \int^t_0 \left[ \frac{\Gamma(\beta)}{(t-s)^{\beta}} 
	\right]\, \dot{\varepsilon}(s)\,ds = \frac{\mathbb{E}}{\Gamma(1-\beta)} 
	\int^t_0 
	\frac{\dot{\varepsilon}(s)}{(t-s)^{\beta}} \,ds \nonumber \\
	& = \mathbb{E}\, {}^C_0 \mathcal{D}^{\beta}_t \left(\varepsilon \right). 
	\label{eq:eqv_improper_caputo}
	\end{align}}
{Substituting} (\ref{eq:eqv_improper_caputo}) into 
(\ref{eq:proofSB_3}), and the result into (\ref{eq:CD-2}), we obtain:
\begin{equation}\label{eq:proofSB_4}
\left[ \tau - \mathbb{E}\, {}^C_0 \mathcal{D}^{\beta}_t \left(\varepsilon 
\right)\right] \dot{\varepsilon} + \int^\infty_0 \frac{\tilde{E}(z)}{z} 
\left(\int^t_0 \exp\left(-\frac{t-s}{z}\right)\dot{\varepsilon}(s) ds\right)^2 
dz \ge 0.
\end{equation}
Since the strain rate $\dot{\varepsilon}$ {is} arbitrary, we 
{set} the argument inside brackets {to zero 
	without violating the above} inequality, {to obtain} the 
stress-strain relationship for the SB model:
\begin{equation*}
\tau = \mathbb{E}\, {}^C_0 \mathcal{D}^{\beta}_t \left(\varepsilon \right).
\end{equation*}
Furthermore, the remainder of (\ref{eq:proofSB_4}) represents an internal 
positive mechanical dissipation, given by:
\begin{equation*}
\mathcal{D}_{mech}(\varepsilon) = \int^\infty_0 \frac{\tilde{E}(z)}{z} 
\left(\int^t_0 \exp\left(-\frac{t-s}{z}\right)\dot{\varepsilon}(s) ds\right)^2 
dz \ge 0,
\end{equation*}
where the above inequality holds, {since} $z$ and $\tilde{E}(z)$ 
are positive.

\section{Proof of Theorem \ref{thm:positive_D}}
\label{Ap:positive_D}

We recall the mechanical dissipation (\ref{eq:DMech}):
\begin{equation}\label{eq:DMech1}
\tau \dot{\varepsilon}^{vp} - R \dot{\alpha} - Y \dot{D} + 
(1-D)\left(\mathcal{D}^{ve}_{mech} + \mathcal{D}^{vp}_{mech}\right) \ge 0,
\end{equation}
where we must prove that the above inequality holds. Substituting 
(\ref{eq:evol_vp_damage}), (\ref{eq:evol_hardening_damage}) into 
(\ref{eq:DMech}) yields:
\begin{equation*}
\tau \mathrm{sign}(\tau)\dot{\gamma} - R \dot{\gamma} - Y \dot{D} + 
(1-D)\left(\mathcal{D}^{ve}_{mech} + \mathcal{D}^{vp}_{mech}\right) \ge 0,
\end{equation*}
Rearranging the above equation, we obtain:
\begin{equation*}
\left[\lvert \tau \rvert - \left( (1-D)\tau^Y + R\right) \right] 
\dot{\gamma} - Y \dot{D} + 
(1-D)\left(\tau^Y \dot{\gamma} + \mathcal{D}^{ve}_{mech} + 
\mathcal{D}^{vp}_{mech}\right) \ge 0,
\end{equation*}
where {the} first term {is is related to the 
	persistency} condition \cite{Simo1998}:
\begin{equation*}
\left[\lvert \tau \rvert - \left( (1-D)\tau^Y + R\left( \alpha \right) \right) 
\right] \dot{\gamma} = f(\tau, \alpha, D) \dot{\gamma} = 0,
\end{equation*}
{and} therefore,
\begin{equation}\label{eq:inequality_proof}
\left(1-D\right)\left(\tau^Y \dot{\gamma} + \mathcal{D}^{ve}_{mech} + 
\mathcal{D}^{vp}_{mech}\right) - Y \dot{D} \ge 0.
\end{equation}
We check the positiveness for each term of the above inequality. 
{For the first term, since the damage is always positive, so is 
	$(1-D)$.} Also, we have $\tau^Y > 0$ and $\dot{\gamma} \ge 0$ 
\cite{Simo1998}. {From Lemma \ref{lemma:stress_strain} the 
	mechanical dissipations $\mathcal{D}^{ve}_{mech}$ and 
	$\mathcal{D}^{vp}_{mech}$ 
	are also positive. For the second term, $-Y$ is positive and so is
	$\dot{D}$, since $D$ is a monotonically increasing function}. Therefore, 
inequality (\ref{eq:inequality_proof}) holds, 
and thus the developed model is thermodynamic admissible.

\section{Convexity of the Yield Function}
\label{Ap:Conxevity}

\begin{proof} Recalling (\ref{eq:Yield_function}), we have
	\begin{align*}
	f\left(\tau, \alpha, D \right) & := \lvert \tau \rvert - \left[ (1-D)\tau^Y 
	+ R 
	\right],
	\end{align*}
	where $R(\alpha,D)=\left(1-D\right) \left[\mathbb{K}\, {}^C_0 
	\mathcal{D}^{\beta_K}_t \left(\alpha \right) + H\alpha \right]$. We fix $D$ 
	{since we are interested in showing the convexity of $f$ with 
		respect to $\tau$ and $R$. Let} $x_1 = \left(\tau_1, R_1\right)$, $x_2 
		= 
	\left(\tau_2, R_2\right)$, $\xi 
	\in [0,1]$, with $R_i = (\alpha_i,D) = \mathbb{K}\, {}^C_0 
	\mathcal{D}^{\beta_K}_t \left(\alpha \right)\big|_{\alpha = \alpha_i} + 
	H\alpha_i$. Therefore, we have:
	{\small
		\begin{align*}
		f\left(\xi x_1 + (1-\xi) x_2\right) = & \vert \xi \tau_1 + (1-\xi) 
		\tau_2 
		\vert - (1-D)\tau^Y - \xi R_1 - (1-\xi)R_2 , \\
		= & \vert \xi \tau_1 + (1-\xi)\tau_2\vert - \xi \left[(1-D) \tau_Y + 
		R_1\right] \\
		& - (1-\xi) \left[(1-D)\tau^Y + R_2\right], \\
		\le & \xi \left\lbrace\vert\tau_1\vert - \left[(1-D)\tau^Y + R_1 
		\right]\right\rbrace \textrm{(by Jensen inequality)} \\
		& + (1-\xi) \left\lbrace\vert\tau_2\vert - 
		\left[(1-D)\tau^Y + R_2 \right]\right\rbrace, \\
		= & \xi f(x_1) + (1-\xi) f(x_2).
		\end{align*}}
\end{proof}

\section{Local Truncation Error for the Free-Energy Discretization}
\label{Ap:Truncation}

We prove the local truncation error (\ref{eq:r_bound}) for 
the discretized Helmholtz free-energy density. Before we prove it, we need 
the following result.
\begin{lemma}\label{lemma1}
	Let $\beta \in (0,1)$, then
	\begin{align}
	\int_{t_i}^{t_{i+1}} \left[(t_{n+1} - s)^{-\beta} - (2t_{n+1} - 
	s)^{-\beta}\right] ds \le C_1\Delta t^{1-\beta},~~~0 \le i \le n,
	\end{align}
	where $C_1$ is a constant independent of $\Delta t$.
\end{lemma}
\begin{proof}
	We can obtain
	{\small
		\begin{align*}
		&\int_{t_i}^{t_{i+1}} \left[(t_{n+1} - s)^{-\beta} - (2t_{n+1} - 
		s)^{-\beta}\right] ds \\
		= & -\frac{(t_{n+1} - s)^{1-\beta}}{1-\beta} 
		\Bigg|_{t_i}^{t_{i+1}} +\frac{(2 t_{n+1} - s)^{1-\beta}}{1-\beta} 
		\Bigg|_{t_i}^{t_{i+1}} \\
		= & \frac{\Delta t^{1-\beta}}{1-\beta} \left[ (n+1-i)^{1-\beta} - 
		(n-i)^{1-\beta} + (2n+1-i)^{1-\beta} - (2n+2-i)^{1-\beta}\right].
		\end{align*}
	}
	Since
	\begin{align*}
	(n+1)^{1-\beta} - n^{1-\beta} = (1-\beta)\int_{-1}^0 (n-s)^{-\beta} ds \le 
	(1-\beta) n^{-\beta},~~~n \ge 1,
	\end{align*}
	then, when $0 \le i \le n-1$, we have
	{\small
		\begin{align*}
		\int_{t_i}^{t_{i+1}} \left[(t_{n+1} - s)^{-\beta} - (2t_{n+1} - 
		s)^{-\beta}\right] ds & \le \Delta t^{1-\beta} 
		\left[\frac{1}{(n-i)^{\beta}} 
		- \frac{1}{(2n+1-i)^{\beta}} \right] \\
		& \le \frac{\Delta 
			t^{1-\beta}}{(n-i)^{\beta}},
		\end{align*}
	}
	for $i=n$, it holds that
	\begin{align*}
	\int_{t_n}^{t_{n+1}} \left[(t_{n+1} - s)^{-\beta} - (2t_{n+1} - 
	s)^{-\beta}\right] ds & = \frac{\Delta t^{1-\beta}}{1-\beta} \left[ 1 + 
	(n+1)^{1-\beta} - (n+2)^{1-\beta}\right] \nonumber\\
	& \le \frac{\Delta t^{1-\beta}}{1-\beta} \left[ 1 - 
	\frac{1-\beta}{(n+1)^{\beta}} \right] \le \frac{\Delta 
		t^{1-\beta}}{1-\beta}.
	\end{align*}
	Therefore this lemma is proved. Next, we prove the local truncation error 
	for the free-energy discretization. From
	\begin{align*}
	\rho \psi(\varepsilon_{n+1}) =& \tilde{\mathbb{E}} \int_0^{t_{n+1}} 
	\int_0^{t_{n+1}} \frac{\dot{\varepsilon}(s_1) \dot{\varepsilon}(s_2) 
	}{(2t_{n+1} - s_1 -s_2)^{\beta}} ds_1 ds_2 \nonumber\\
	=& \tilde{\mathbb{E}} \sum_{i,j=0}^n \int_{t_i}^{t_{i+1}} 
	\int_{t_j}^{t_{j+1}} \frac{\Delta \varepsilon_{i+1} 
		\Delta \varepsilon_{j+1}}{\Delta t^2(2t_{n+1} - s_1 
		-s_2)^{\beta}} ds_1 ds_2 + \tilde{r}_{\Delta t}^{n+1},
	\end{align*}
	with $\tilde{\mathbb{E}} = \frac{\mathbb{E}}{2\Gamma(1-\beta)}$ and $\Delta 
	\varepsilon_{k+1} = \varepsilon_{k+1} - \varepsilon_k$. We know 
	that
	{\small
		\begin{align*}
		\left| \tilde{r}_{\Delta t}^{n+1} \right|=& \Bigg|\tilde{\mathbb{E}} 
		\sum_{i,j=0}^n \int_{t_i}^{t_{i+1}} \int_{t_j}^{t_{j+1}} (2t_{n+1} - 
		s_1 
		-s_2)^{-\beta} \left[\dot{\varepsilon}(s_1) \dot{\varepsilon}(s_2) -  
		\frac{\Delta \varepsilon_{i+1} \Delta \varepsilon_{j+1}}{\Delta 
			t^2}\right] 
		ds_1 ds_2\Bigg| \\
		=& \Bigg|\tilde{\mathbb{E}} \sum_{i,j=0}^n \int_{t_i}^{t_{i+1}} 
		\int_{t_j}^{t_{j+1}} (2t_{n+1} - s_1 -s_2)^{-\beta} 
		\left[\dot{\varepsilon}(s_1) \dot{\varepsilon}(s_2) - 
		\dot{\varepsilon}(s_1) \frac{\Delta \varepsilon_{j+1}}{\Delta t} 
		\right.\\
		& \left. + \dot{\varepsilon}(s_1) \frac{\Delta 
			\varepsilon_{j+1}}{\Delta t} 
		- \frac{\Delta \varepsilon_{i+1} \Delta \varepsilon_{j+1}}{\Delta 
			t^2}\right] ds_1 ds_2\Bigg| \\
		=& \Bigg|\tilde{\mathbb{E}} \sum_{i,j=0}^n \int_{t_i}^{t_{i+1}} 
		\int_{t_j}^{t_{j+1}} (2t_{n+1} - s_1 -s_2)^{-\beta} 
		\left[\dot{\varepsilon}(s_1) \left(\dot{\varepsilon}(s_2) - 
		\frac{\Delta \varepsilon_{j+1}}{\Delta t} \right) \right. \\
		& \left. + \frac{\Delta \varepsilon_{j+1}}{\Delta t} \left( 
		\dot{\varepsilon}(s_1) - \frac{\Delta \varepsilon_{i+1}}{\Delta 
			t}\right) 
		\right] ds_1 ds_2\Bigg| \nonumber\\
		\le& \tilde{\mathbb{E}} \Bigg|\sum_{i,j=0}^n \int_{t_i}^{t_{i+1}} 
		\int_{t_j}^{t_{j+1}} (2t_{n+1} - s_1 -s_2)^{-\beta} 
		\dot{\varepsilon}(s_1) 
		\left(\dot{\varepsilon}(s_2) - \frac{\Delta \varepsilon_{j+1}}{\Delta 
			t} 
		\right) ds_1 ds_2\Bigg| \\
		& + \tilde{\mathbb{E}} \Bigg|\sum_{i,j=0}^n \frac{\Delta 
			\varepsilon_{j+1}}{\Delta t} \int_{t_i}^{t_{i+1}} 
		\int_{t_j}^{t_{j+1}} 
		(2t_{n+1} - s_1 -s_2)^{-\beta} \left( \dot{\varepsilon}(s_1) - 
		\frac{\Delta \varepsilon_{i+1}}{\Delta t}\right) ds_1 ds_2\Bigg| 
		\\
		:= & I_1 + I_2,
		\end{align*}
	}
	where 
	{\small$$I_1 = \tilde{\mathbb{E}} \Bigg|\sum_{i,j=0}^n \int_{t_i}^{t_{i+1}} 
		\int_{t_j}^{t_{j+1}} (2t_{n+1} - s_1 -s_2)^{-\beta} 
		\dot{\varepsilon}(s_1) 
		\left(\dot{\varepsilon}(s_2) - \frac{\Delta \varepsilon_{j+1}}{\Delta 
		t} 
		\right) ds_1 ds_2\Bigg|,$$}
	and
	{\small$$I_2 = \tilde{\mathbb{E}} \Bigg|\sum_{i,j=0}^n \frac{\Delta 
			\varepsilon_{j+1}}{\Delta t} \int_{t_i}^{t_{i+1}} 
		\int_{t_j}^{t_{j+1}} (2t_{n+1} - s_1 -s_2)^{-\beta} \left( 
		\dot{\varepsilon}(s_1) - \frac{\Delta \varepsilon_{i+1}}{\Delta 
			t}\right) ds_1 ds_2\Bigg|.$$}
	
	Assume $\varepsilon(t) \in C^2[0,T]$, then one can obtain that:
	$$\varepsilon(t) \le C_2,~~~\dot{\varepsilon}(t) \le C_3,~t \in [0,T].$$
	On each small interval $[t_i,t_{i+1}]~(0 \le i \le n)$, denoting the linear 
	interpolation function of $\varepsilon(t)$ as $\Pi_i \varepsilon(t)$:
	\begin{align*}
	\Pi_i \varepsilon(t) = \frac{t-t_{i+1}}{t_i - t_{i+1}} \varepsilon_i + 
	\frac{t-t_i}{t_{i+1} - t_i}\varepsilon_{i+1},
	\end{align*}
	it follows from the linear interpolation theory that
	\begin{align*}
	\varepsilon(t) - \Pi_i \varepsilon(t) = \frac{\varepsilon''(\xi_i)}{2} 
	(t-t_i)(t-t_{i+1}) \le c_i \Delta t^2,~~~t \in [t_i,t_{i+1}],~\xi_i \in 
	(t_i,t_{i+1}),
	\end{align*}
	with $0 \le i \le n$, and here $c_i$ is a constant independent of $\Delta 
	t$.
	
	For $I_1$, we have
	{\small
		\begin{align*}
		I_1 = & \tilde{\mathbb{E}} \Bigg|\sum_{i,j=0}^n \int_{t_i}^{t_{i+1}} 
		\dot{\varepsilon}(s_1) \int_{t_j}^{t_{j+1}} (2t_{n+1} - s_1 
		-s_2)^{-\beta} 
		\left[\varepsilon(s_2) - \Pi_j \varepsilon(s_2) \right]' ds_1 
		ds_2\Bigg| \\
		= & \tilde{\mathbb{E}} \Bigg|\sum_{i,j=0}^n \int_{t_i}^{t_{i+1}} 
		\dot{\varepsilon}(s_1) \int_{t_j}^{t_{j+1}} (2t_{n+1} - s_1 
		-s_2)^{-\beta} 
		d \left[\varepsilon(s_2) - \Pi_j \varepsilon(s_2) \right] 
		ds_1\Bigg| \\
		= & \beta \tilde{\mathbb{E}} \Bigg|\sum_{i,j=0}^n \int_{t_i}^{t_{i+1}} 
		\dot{\varepsilon}(s_1) \int_{t_j}^{t_{j+1}} \left[\varepsilon(s_2) - 
		\Pi_j \varepsilon(s_2) \right]  (2t_{n+1} - s_1 -s_2)^{-\beta-1} 
		ds_1 
		ds_2\Bigg| \\
		\le & \beta \tilde{\mathbb{E}} \Bigg|\sum_{i,j=0}^n 
		\int_{t_i}^{t_{i+1}} 
		\dot{\varepsilon}(s_1) \int_{t_j}^{t_{j+1}} c_j \Delta 
		t^2 (2t_{n+1} - s_1 -s_2)^{-\beta-1} ds_1 ds_2\Bigg| \\
		\le & \beta \tilde{\mathbb{E}} C_4 \Delta t^2 \Bigg|\sum_{i=0}^n 
		\int_{t_i}^{t_{i+1}} \dot{\varepsilon}(s_1) \int_{0}^{t_{n+1}} 
		(2t_{n+1} - 
		s_1 -s_2)^{-\beta-1} ds_1 ds_2\Bigg| \\
		= & \tilde{\mathbb{E}} C_4 \Delta t^2 \Bigg|\sum_{i=0}^n 
		\int_{t_i}^{t_{i+1}} 
		\dot{\varepsilon}(s_1) \left[(t_{n+1} - s_1)^{-\beta} - (2t_{n+1} - 
		s_1)^{-\beta}\right] ds_1\Bigg| \\
		\le & \tilde{\mathbb{E}} C_3 C_4 \Delta t^2 \Bigg|\sum_{i=0}^n 
		\int_{t_i}^{t_{i+1}}  \left[(t_{n+1} - s_1)^{-\beta} - (2t_{n+1} - 
		s_1)^{-\beta}\right] ds_1\Bigg|,
		\end{align*}}
	where $C_4=\max\limits_{0\le j \le n} c_j$. For $I_2$, it holds that
	{\small
		\begin{align*}
		I_2 = & \tilde{\mathbb{E}} \Bigg|\sum_{i,j=0}^n \frac{\Delta 
			\varepsilon_{j+1}}{\Delta t} \int_{t_i}^{t_{i+1}} 
		\int_{t_j}^{t_{j+1}} (2t_{n+1} - s_1 -s_2)^{-\beta} \left[ 
		\varepsilon(s_1) 
		- \Pi_i \varepsilon(s_1)\right]' ds_1 ds_2\Bigg| \\
		= & \tilde{\mathbb{E}} \Bigg|\sum_{i,j=0}^n \frac{\Delta 
			\varepsilon_{j+1}}{\Delta t} \int_{t_j}^{t_{j+1}} 
		\int_{t_i}^{t_{i+1}} 
		(2t_{n+1} - s_1 -s_2)^{-\beta} d\left[ \varepsilon(s_1) - \Pi_i 
		\varepsilon(s_1)\right] ds_2\Bigg| \\
		= & \beta \tilde{\mathbb{E}} \Bigg|\sum_{i,j=0}^n \frac{\Delta 
			\varepsilon_{j+1}}{\Delta t} \int_{t_j}^{t_{j+1}} 
		\int_{t_i}^{t_{i+1}} \left[ \varepsilon(s_1) - \Pi_i 
		\varepsilon(s_1)\right] (2t_{n+1} - s_1 -s_2)^{-\beta-1} ds_1 
		ds_2\Bigg| \\
		\le & \beta \tilde{\mathbb{E}} \Bigg|\sum_{i,j=0}^n \frac{\Delta 
			\varepsilon_{j+1}}{\Delta t} \int_{t_j}^{t_{j+1}} 
		\int_{t_i}^{t_{i+1}} c_i \Delta t^2 (2t_{n+1} - s_1 -s_2)^{-\beta-1} 
		ds_1 
		ds_2\Bigg| \\
		\le & \beta \tilde{\mathbb{E}} C_5 \Delta t \Bigg|\sum_{j=0}^n 
		\Delta \varepsilon_{j+1} \int_{t_j}^{t_{j+1}} 
		\int_{0}^{t_{n+1}} (2t_{n+1} - s_1 -s_2)^{-\beta-1} ds_1 ds_2\Bigg| \\
		= & \tilde{\mathbb{E}} C_5 \Delta t \Bigg|\sum_{j=0}^n 
		\Delta \varepsilon_{j+1} \int_{t_j}^{t_{j+1}} \left[(t_{n+1} - 
		s_2)^{-\beta} - 
		(2t_{n+1} - s_2)^{-\beta}\right] ds_2\Bigg|,
		\end{align*}}
	where $C_5 = \max\limits_{0\le i \le n} c_i$. Then, it follows from Lemma 
	\ref{lemma1} that
	\begin{align*}
	I_1 + I_2 \le & \frac{n \mathbb{E} C_1 C_3 C_4 \Delta 
		t^{3-\beta}}{2\Gamma(1-\beta)} + \frac{\mathbb{E} C_1 C_5 \Delta 
		t^{2-\beta}}{2\Gamma(1-\beta)} \Bigg|\sum_{j=0}^n \Delta 
	\varepsilon_{j+1} 
	\Bigg| \\
	=& \frac{\mathbb{E} C_1 C_3 C_4 T \Delta t^{2-\beta}}{2\Gamma(1-\beta)} + 
	\frac{\mathbb{E} C_1 C_5 \Delta t^{2-\beta}}{2\Gamma(1-\beta)} 
	\big|\varepsilon_{n+1} - \varepsilon_0\big| \\
	\le & \frac{\mathbb{E} C_1(C_3 C_4 T + 2C_2 C_5)}{2\Gamma(1-\beta)} \Delta 
	t^{2-\beta}.
	\end{align*}
\end{proof}

\section*{Acknowledgments}
This work was supported by the ARO YIP Award number (W911NF-19-1-0444), the NSF 
Award 
(DMS-1923201), also 
partially by the MURI/\\ARO Award (W911NF-15-1-0562) and the AFOSR YIP Award 
(FA9550-17-1-0150). M. D'Elia was supported by Sandia National Laboratories 
(SNL), SNL is a multimission laboratory managed and operated by National 
Technology and Engineering Solutions of Sandia, LLC., a wholly owned subsidiary 
of Honeywell International, Inc., for the U.S. Department of Energys National 
Nuclear Security Administration contract number DE-NA0003525. This paper 
describes objective technical results and analysis. Any subjective views or 
opinions that might be expressed in the paper do not necessarily represent the 
views of the U.S. Department of Energy or the United States Government. 
SAND2019-14071 R.

\bibliographystyle{siamplain}
\bibliography{Failure_references}

\begin{thebibliography}{10}

\bibitem{Adolfsson2005}
{\sc K.~Adolfsson, M.~Enelund, and P.~Olsson}, {\em On the fractional order
  model of viscoelasticity}, Mech. Time-Depend. Mat., 9 (2005), pp.~15--34.

\bibitem{Alfano2017}
{\sc G.~Alfano and M.~Musto}, {\em Thermodynamic derivation and damage
  evolution for a fractional cohesive zone model}, J. Eng. Mech., 143 (2017).

\bibitem{Bejan2016}
{\sc A.~Bejan}, {\em Advanced Engineering Thermodynamics, Fourth Edition},
  Wiley, 2016.

\bibitem{Blair1947}
{\sc G.~Blair, B.~Veinoglou, and B.~Caffyn}, {\em Limitations of the
  \uppercase{N}ewtonian time scale in relation to non-equilibrium rheological
  states and a theory of quasi-properties}, Proc. R. Soc. Lond. A, 189 (1947),
  pp.~69--87.

\bibitem{Blair1943}
{\sc S.~Blair and F.~Coppen}, {\em The estimation of firmness in soft
  materials}, Amer. J. Psychol., 56 (1943), pp.~234--246.

\bibitem{boldrini2016non}
{\sc J.~Boldrini, E.~B. de~Moraes, L.~Chiarelli, F.~Fumes, and M.~Bittencourt},
  {\em A non-isothermal thermodynamically consistent phase field framework for
  structural damage and fatigue}, Comput. Methods Appl. Mech. Eng., 312 (2016),
  pp.~395--427.

\bibitem{Bonadkar2016}
{\sc N.~Bonadkar, R.~Gerum, M.~Kuhn, M.~Sporer, A.~Lippert, W.~Schneider,
  K.~Aifantis, and B.~Fabry}, {\em Mechanical plasticity of cells}, Nat.
  Mater., 15 (2016), pp.~1090 -- 1094.

\bibitem{Bouchard2011}
{\sc P.-O. Bouchard, L.~Bourgeon, S.~Fayolle, and K.~Mocellin}, {\em An
  enhanced lemaitre model formulation for materials processing damage
  computation}, Int. J. Mater. Form., 4 (2011), pp.~299--315.

\bibitem{Burlon2014}
{\sc A.~Burlon, F.~Pinnola, and M.~Zingales}, {\em A numerical assessment of
  the free energy function for fractional-order relaxation}, in ICFDA`14 -
  International Conference on Fractional Differentiation and its Applications,
  Catania, Italy, 2014.

\bibitem{Cao2013}
{\sc T.-S. Cao, A.~Gaillac, P.~Montmitonnet, and P.-O. Bouchard}, {\em
  Identification methodology and comparison of phenomenological ductile damage
  models via hybrid numerical–experimental analysis of fracture experiments
  conducted on a zirconium alloy}, Int. J. Solids Struct., 50 (2013),
  pp.~3984--3999.

\bibitem{Cao2016}
{\sc W.~Cao, F.~Zeng, Z.~Zhongqiang, and G.~Karniadakis}, {\em
  Implicit-explicit difference schemes for nonlinear fractional differential
  equations with nonsmooth solutions}, SIAM J. Sci. Comput., 38 (2016),
  pp.~A3070--A3093.

\bibitem{Caputo2015}
{\sc M.~Caputo and M.~Fabrizio}, {\em Damage and fatigue described by a
  fractional derivative model}, J. Comp. Phys., 293 (2015), pp.~400 -- 408.

\bibitem{deSouzaNeto2008}
{\sc E.~de~Souza~Neto, D.~Peri\'{c}, and D.~Owen}, {\em Computational Methods
  for Plasticity: Theory and Applications}, John Wiley \& Sons, 2008.

\bibitem{DElia2013}
{\sc M.~D'Elia and M.~Gunzburger}, {\em The fractional laplacian operator on
  bounded domains as a special case of the nonlocal diffusion operator},
  Comput. Math. Appl., 66 (2013), pp.~1245--1260.

\bibitem{Deseri2014}
{\sc L.~Deseri, M.~DiPaola, and M.~Zingales}, {\em Free energy and states of
  fractional-order hereditariness}, Int. J. Solids Struct., 51 (2014), pp.~3156
  -- 3167.

\bibitem{Diethelm2004}
{\sc K.~Diethelm, N.~J. Ford, and A.~D. Freed}, {\em Detailed error analysis
  for a fractional \uppercase{A}dams method}, Numer. Algorithms, 36 (2004),
  pp.~31--52.

\bibitem{Dekic2017}
{\sc D.~Doli\'{c}anin-\DJ{}eki\'{c}}, {\em On a new class of constitutive
  equations for linear viscoelastic body}, Fract. Calc. Appl. Anal., 20 (2017),
  pp.~521--536.

\bibitem{Du2013}
{\sc Q.~Du, M.~Gunzburger, R.~B. Lehoucq, and K.~Zhou}, {\em A nonlocal vector
  calculus, nonlocal volume--constrained problems, and nonlocal balance laws},
  Math. Models Methods Appl. Sci., 23 (2013), pp.~493--540.

\bibitem{Du2017}
{\sc Q.~Du, J.~Yang, and Z.~Zhou}, {\em Analysis of a nonlocal-in-time
  parabolic equation}, Discrete Contin. Dyn. Sys. Ser. B, 22 (2017),
  pp.~339--368.

\bibitem{Fabrizio2014}
{\sc M.~Fabrizio}, {\em Fractional rheological models for thermomechanical
  systems. dissipation and free energies}, Fract. Calc. Appl. Anal., 17 (2014),
  pp.~206 -- 223.

\bibitem{Feng2007}
{\sc L.~Feng and V.~Linetsky}, {\em Pricing options in jump-diffusion models:
  An extrapolation approach}, Oper. Res., 56 (2008), pp.~304--325.

\bibitem{Giraldo-Londono2019}
{\sc O.~Giraldo-Londo\~{n}o, G.~Paulino, and W.~Buttlar}, {\em Fractional
  calculus derivation of a rate-dependent \uppercase{PPR}-based cohesive
  fracture model: theory, implementation, and numerical results}, Int. J.
  Fract., 216 (2019), pp.~1--29.

\bibitem{Giusti2017}
{\sc A.~Giusti}, {\em On infinite order differential operators in fractional
  viscoelasticity}, Fract. Calc. Appl. Anal., 20 (2017), pp.~854--867.

\bibitem{Gurtin2010}
{\sc M.~Gurtin, E.~Fried, and L.~Anand}, {\em The mechanics and thermodynamics
  of continua}, Cambridge University Press, 2010.

\bibitem{Hei2018}
{\sc X.~Hei, W.~Chen, G.~Pang, R.~Xiao, and C.~Zhang}, {\em A new
  visco-elasto-plastic model via time-space derivative.}, Mech. Time-Depend.
  Mater., 22 (2018), pp.~129--141.

\bibitem{Jaishankar2013}
{\sc A.~Jaishankar and G.~McKinley}, {\em Power-law rheology in the bulk and at
  the interface: quasi-properties and fractional constitutive equations}, Proc
  R Soc A 469: 20120284,  (2013).

\bibitem{Kang2015}
{\sc J.~Kang, F.~Zhou, C.~Liu, and Y.~Liu}, {\em A fractional non-linear creep
  model for coal considering damage effect and experimental validation}, Int.
  J. Nonlin. Mech., 76 (2015), pp.~20--28.

\bibitem{kharazmi2018operatorbased}
{\sc E.~Kharazmi and M.~Zayernouri}, {\em Operator-based uncertainty
  quantification of stochastic fractional pdes}, ASME J. Verif. Valid. Uncert.
  (in press),  (2018).

\bibitem{kharazmi2017FSEM}
{\sc E.~Kharazmi and M.~Zayernouri}, {\em Fractional sensitivity equation
  method: Applications to fractional model construction}, J. Sci. Comput., 80
  (2019), pp.~110--140.

\bibitem{Lemaitre1996}
{\sc J.~Lemaitre}, {\em A course on damage mechanics}, Springer, 1996.

\bibitem{Lemaitre2005}
{\sc J.~Lemaitre and R.~Desmorat}, {\em Engineering Damage Mechanics},
  Springer, 2005.

\bibitem{Lin2007}
{\sc Y.~Lin and C.~Xu}, {\em Finite difference/spectral approximations for the
  time-fractional diffusion equation}, J. Comp. Phys., 225 (2007),
  pp.~1533--1552.

\bibitem{Lion1997}
{\sc A.~Lion}, {\em On the thermodynamics of fractional damping elements},
  Continuum Mech. Thermodyn., 9 (1997), pp.~83 -- 96.

\bibitem{Lubich1986}
{\sc C.~Lubich}, {\em Discretized fractional calculus}, SIAM J. Math. Anal., 17
  (1986), pp.~704--719.

\bibitem{Lubich2002}
{\sc C.~Lubich and A.~Sch\"{a}dle.}, {\em Fast convoluion for nonreflecting
  boundary conditions}, SIAM J. Sci. Comput., 24 (2002), pp.~161--182.

\bibitem{Mashayekhi2019Fractal}
{\sc S.~Mashayekhi, Y.~Hussaini, and W.~Oates}, {\em A physical interpretation
  of fractional viscoelasticity based on the fractal structure of media: Theory
  and experimental validation}, J. Mech. Phys. Solids, 128 (2019),
  pp.~137--150.

\bibitem{Magin2010}
{\sc F.~Meral, T.~Royston, and R.~Magin}, {\em Fractional calculus in
  viscoelasticity: An experimental study}, Commun. Nonlinear Sci. Numer.
  Simul., 15 (2010), pp.~939--945.

\bibitem{naghibolhosseini2015estimation}
{\sc M.~Naghibolhosseini}, {\em Estimation of outer-middle ear transmission
  using \uppercase{DPOAE}s and fractional-order modeling of human middle ear},
  PhD thesis, City University of New York, NY., 2015.

\bibitem{Naghibolhosseini2018}
{\sc M.~Naghibolhosseini and G.~Long}, {\em Fractional-order modelling and
  simulation of human ear}, Int. J. Comput. Math., 95 (2018), pp.~1257--1273.

\bibitem{Nasholm2013Wave}
{\sc S.~N\"{a}sholm and S.~Holm}, {\em On a fractional zener elastic wave
  equation}, Fract. Calc. Appl. Anal., 16 (2013), pp.~26--60.

\bibitem{Podlubny99}
{\sc I.~Podlubny}, {\em Fractional Differential Equations}, San Diego, CA, USA:
  Academic Press, 1999.

\bibitem{Roux2015}
{\sc E.~Roux and P.-O. Bouchard}, {\em On the interest of using full field
  measurements in ductile damage model calibration}, Int. J. Solids Struct., 72
  (2015), pp.~50--62.

\bibitem{Samiee2019FSSM}
{\sc M.~Samiee, A.~Akhavan-Safaei, and M.~Zayernouri}, {\em A fractional
  subgrid-scale model for turbulent flows: Theoretical formulation and a priori
  study}, arXiv preprint arXiv:1909.09943,  (2019).

\bibitem{Samiee2016}
{\sc M.~Samiee, M.~Zayernouri, and M.~M. Meerschaert}, {\em A unified spectral
  method for \uppercase{FPDE}s with two-sided derivatives; part \uppercase{i}:
  A fast solver}, J. Comp. Phys., 385 (2019), pp.~225--243.

\bibitem{samiee2017unified}
{\sc M.~Samiee, M.~Zayernouri, and M.~M. Meerschaert}, {\em A unified spectral
  method for \uppercase{FPDE}s with two-sided derivatives; part \uppercase{ii}:
  Stability, and error analysis}, J. Comp. Phys., 385 (2019), pp.~244--261.

\bibitem{Schiessel1995}
{\sc H.~Schiessel}, {\em Generalized viscoelastic models: their fractional
  equations with solutions}, J. Phys. A: Math. Gen., 28 (1995), pp.~6567--6584.

\bibitem{Schiessel1993}
{\sc H.~Schiessel and A.~Blumen}, {\em Hierarchical analogues to fractional
  relaxation equations}, J. Phys. A: Math. Gen., 26 (1993), pp.~5057--5069.

\bibitem{Simo1987}
{\sc J.~Simo}, {\em On a fully three-dimensional finite-strain viscoelastic
  damage model: Formulation and computational aspects}, Comp. Methods Appl.
  Mech. Eng., 60 (1987), pp.~153--173.

\bibitem{Simo1998}
{\sc J.~Simo and T.~Hughes}, {\em Computational inelasticity}, Springer, 1998.

\bibitem{Sumelka2014VP}
{\sc W.~Sumelka}, {\em Fractional viscoplasticity}, Mech. Res. Commun., 56
  (2014), pp.~31--36.

\bibitem{Sumelka2019Soil}
{\sc Y.~Sun and W.~Sumelka}, {\em Fractional viscoplastic model for soils under
  compression}, Acta Mech., 230 (2019), pp.~3365--3377.

\bibitem{Suzuki2014}
{\sc J.~Suzuki and P.~Mu\~{n}oz Rojas}, {\em Transient analysis of
  geometrically non-linear trusses considering coupled plasticity and damage},
  Tenth World Congress on Computational Mechanics, 1 (2014), pp.~322--341.

\bibitem{Suzuki2018Singularity}
{\sc J.~Suzuki and M.~Zayernouri}, {\em An automated singularity-capturing
  scheme for fractional differential equations.}, arXiv:1810.12219,  (2018).

\bibitem{Suzuki2016}
{\sc J.~Suzuki, M.~Zayernouri, M.~Bittencourt, and G.~Karniadakis}, {\em
  Fractional-order uniaxial visco-elasto-plastic models for structural
  analysis}, Comput. Methods Appl. Mech. Eng., 308 (2016), pp.~443 -- 467.

\bibitem{Tang2018}
{\sc H.~Tang, D.~Wang, R.~Huang, X.~Pei, and W.~Chen}, {\em A new rock creep
  model based on variable-order fractional derivatives and continuum damage
  mechanics}, B. Eng. Geol. Environ., 77 (2018), pp.~375--383.

\bibitem{varghaei2019Vibration}
{\sc P.~Varghaei, E.~Kharazmi, J.~Suzuki, and M.~Zayernouri}, {\em Vibration
  analysis of geometrically nonlinear and fractional viscoelastic cantilever
  beams}, arXiv preprint arXiv:1909.02142,  (2019).

\bibitem{Vikram2010}
{\sc M.~Vikram, A.~Baczewski, B.~Shanker, and L.~Kempel}, {\em Accelerated
  cartesian expansion (\uppercase{ACE}) based framework for the rapid
  evaluation of diffusion, lossy wave, and klen-gordon potentials}, J. Comp.
  Phys., 229 (2010), pp.~9119--9134.

\bibitem{Wegst20151}
{\sc U.~Wegst, H.~Bai, E.~Saiz, A.~Tomsia, and R.~Ritchie}, {\em Bioinspired
  structural materials}, Nat. Mater., 14 (2015), pp.~23--36.

\bibitem{Xiao2017}
{\sc R.~Xiao, H.~Sun, and W.~Chen}, {\em A finite deformation fractional
  viscoplastic model for the glass transition behavior of amorphous polymers},
  Int. J. Nonlin. Mech., 93 (2017), pp.~7--14.

\bibitem{Zayernouri2016MS}
{\sc M.~Zayernouri and A.~Matzavinos}, {\em Fractional
  \uppercase{A}dams-\uppercase{B}ashforth/\uppercase{M}oulton methods: An
  application to the fractional \uppercase{K}eller-\uppercase{S}egel chemotaxis
  system.}, J. Comp. Phys., 317 (2016), pp.~1--14.

\bibitem{Zheng2017}
{\sc F.~Zeng, I.~Turner, and K.~Burrage}, {\em A stable fast time-stepping
  method for fractional integral and derivative operators}, J. Sci. Comput.,
  (2018).

\bibitem{Zhang2014creep}
{\sc C.~Zhang, Z.~Zhu, S.~Zhu, D.~He, D.~Zhu, J.~Liu, and S.~Meng}, {\em
  Nonlinear creep damage constitutive model of concrete based on fractional
  calculus theory}, Materials, 12 (2014), pp.~1 -- 14.

\bibitem{Zhou2019IMEX}
{\sc Y.~Zhou, J.~Suzuki, C.~Zhang, and M.~Zayernouri}, {\em Fast
  \uppercase{IMEX} time integration of nonlinear stiff fractional differential
  equations}, arXiv preprint arXiv:1909.04132,  (2019).

\end{thebibliography}

\end{document}